\documentclass[journal]{IEEEtran}
\ifCLASSINFOpdf
\usepackage[pdftex]{graphicx}
\graphicspath{{../pdf/}{../jpeg/}}
\DeclareGraphicsExtensions{.pdf,.jpeg,.png}
\else

\usepackage{epsfig}
\usepackage{graphicx}
\usepackage[cmex10]{amsmath}
\graphicspath{{../eps/}}
\allowdisplaybreaks

\fi

\usepackage{color}
\usepackage{amssymb}
\usepackage{amsmath}
\usepackage{amsthm}
\usepackage{subfig}
\usepackage{bm}
\theoremstyle{plain}
\newtheorem{assm}{Assumption}
\newtheorem{clam}{Claim}
\newtheorem{lema}{Lemma}
\newtheorem{thme}{Theorem}
\newtheorem{remark}{Remark}
\newtheorem{prope}{Proposition}
\theoremstyle{definition}

\newcommand\BibTeX{{\rmfamily B\kern-.05em \textsc{i\kern-.025em b}\kern-.08em
T\kern-.1667em\lower.7ex\hbox{E}\kern-.125emX}}
\begin{document}

\title{Output-Positive Adaptive Control\\ of Hyperbolic PDE-ODE Cascades}

\author{Ji~Wang,~\IEEEmembership{Member,~IEEE},
       and Miroslav~Krstic,~\IEEEmembership{Fellow,~IEEE}
\thanks{
J. Wang is with the Department of Automation, Xiamen University, Xiamen,
Fujian 361005, China (e-mail:jiwang@xmu.edu.cn).

M. Krstic is with the Department of Mechanical and Aerospace Engineering, University of California, San Diego, La Jolla, CA 92093-0411, USA(e-mail: krstic@ucsd.edu).
}}

\maketitle
\begin{abstract}
In this paper, we propose a new adaptive Control Barrier Function (aCBF) method to design the output-positive adaptive control law for a hyperbolic PDE-ODE cascade with parametric uncertainties. This method employs the recent adaptive
control approach with batch least-squares identification (BaLSI,
pronounced ``ballsy'') that completes perfect parameter identification
in finite time and offers a previously unforeseen advantage in
safe control design with aCBF, which we elucidate in this paper. Since the true challenge is exhibited for CBF of a high relative degree, we undertake a control design in this paper for a class of systems that possess a particularly extreme relative degree: $2\times2$ hyperbolic PDEs sandwiched by a strict-feedback nonlinear ODE and a linear ODE, where the unknown coefficients are associated with the PDE in-domain coupling terms and with the input signal of the distal ODE. The designed output-positive adaptive controller guarantees the positivity of the output signal that is the furthermost state from the control input as well as the exponential regulation of the overall plant state to zero. The effectiveness of the proposed method is illustrated by numerical simulation.
\end{abstract}


\section{Introduction}
The safe control with adaptive Control Barrier Function (aCBF), employing conventional continuous infinite-time adaptation, requires that the initial conditions be restricted to a subset of the safe set due to parametric uncertainty, where the safe set is shrunk in inverse proportion to the adaptation gain. The recent regulation-triggered adaptive control approach with batch least-squares identification (BaLSI, pronounced ``ballsy'') completes perfect parameter identification in finite time and offers a previously unforeseen advantage in aCBF-based safe control, which we elucidate in this paper.
\subsection{Physical motivation}\label{sec:UAV}
This work is motivated by control of UAVs with hanging loads for delivery, where the suspended objects or
payloads are tied to the bottom of a cable, of which
the other end is hanged to a multi-rotor UAV, i.e., a structure of UAV-cable-payload. The control task is to place the payload at the setpoint that is in a narrow space and avoid collision with the surrounding environment by regulating the UAV, i.e., designing the speed of rotors of multi-rotor UAVs. The mathematical model of the structure of UAV-cable-payload is a nonlinear ODE-PDE-ODE sandwich system \cite{J2020delay}. In detail, the oscillation dynamics of cables are modeled by wave PDEs \cite{J2018Balancing}, the
lumped tip payload at one end of the cable is described by a linear ODE, and the multi-rotor UAV dynamics are modeled as a nonlinear ODE considering the non-negligible couplings in its motion dynamics. The control input of the overall sandwich system, i.e., the speed of rotors of multi-rotor UAVs, is the control input applied at the nonlinear ODE. For the sake of order reduction, the wave PDEs
can usually be converted to  heterodirectional coupled transport
PDEs via Riemann transformations \cite{J2018Balancing}. Therefore, the aforementioned control
task comes down to
the following theoretical problem: safe boundary control of heterodirectional coupled hyperbolic
PDEs sandwiched by nonlinear and linear ODEs, to ensure the exponential regulation of the overall sandwich system and meanwhile guarantee the safety of the distal ODE, i.e., constraining the state furthest from the control input in a safe region. Besides, considering the uncertainties that often appear in the cable material characteristics and the hanging load, we further require a safe adaptive control scheme, which allows uncertain parameters existing in the PDE and distal ODE and still guarantees the stability and safety mentioned above.
\subsection{Boundary control of coupled hyperbolic
PDE systems}
In addition to the UAV with hanging load mentioned above, heterodirectional coupled hyperbolic PDEs are also ubiquitous in numerous physical models involving the water level dynamics \cite{Prieur2018Boundary}, traffic flow \cite{Goatin2006},  string-actuated mechanisms \cite{J2022PDE}, and so on.
Basic boundary stabilization problem of $2\times2$ linear
hyperbolic PDEs was addressed in \cite{Coron2013Local,Vazquez2011Backstepping} by backstepping,
which was further extended to boundary control of a $n+1$ hyperbolic system in \cite{Meglio2013Stabilization}, and a more general $n+m$ hyperbolic system, where the
number of PDEs in either direction is arbitrary, in \cite{Hu2016Control}.
Incorporating ODEs, control of the heterodirectional coupled hyperbolic  PDE-ODE systems was addressed in \cite{Meglio2017Stabilization,J2018Balancing,Deutscher2018Output1} as well.

The actuator in most boundary controlled PDEs does not act directly and instantly. The actuator has its own considerable inertia, namely, its own lumped-parameter dynamics modeled by {an ODE}. Consequently,  the overall system that arises in control of heterodirectional coupled hyperbolic PDE-ODE models is often in a sandwich configuration (ODE-PDE-ODE), where the control input acts on only one of the two ODEs, and not directly on the PDE.
Nominal control designs
of heterodirectional coupled hyperbolic PDEs sandwiched between two ODEs
can be found in \cite{Saba2017,Saba2019,J2017Control,Deutscher2018Output}. Addressing the additional effects in real applications, such as nonlinearities, disturbances, unknown parameters, delays, sampled (or event-triggered) sensing, and robustness, the extended control designs for the sandwich PDEs were presented in \cite{J2020delay,Meglio2019Robust,J2021Output-feedback,Auriol2023Robustification}.
\subsection{CBF Safe Control}
The PDE control designs listed above only pursue stabilization without considering safety--that is, constraining
the state in a safe region.  In many engineering applications,
like autonomous driving, robotics, UAVs, not only stability but also safety
for avoiding collision is vital. The notion of safety was first
introduced in \cite{1977Lamport} and formalized in \cite{1984Lamport}, by L. Lamport.

One way to ensure safety in a dynamic system is by using the Barrier Lyapunov Function (BLF) method. According to \cite{Tee2009}, this method involves constructing a barrier Lyapunov function $B$ and enforcing that $B\le 0$ over the set $C$ that should be contained in the safe set. Recently, a learning-based BLF method was developed in \cite{2023Zhang,2022Zhang} for adaptive control of safety-critical autonomous vehicles, constraining the states within the safety region even during the learning process.  However, a major limitation of the BLF-based method is that while it ensures safety, it also enforces invariance of every level set, which makes it overly strong and conservative. \cite{Ames2019}.

The Control Barrier Functions (CBFs), introduced in  \cite{Ames2016},  turn out to be a powerful tool for the safe control design, where the considered system state can be constrained in a safe region
once the non-negativity of CBF is ensured. To guarantee this in the closed-loop system, the CBF-based control is often accompanied with a quadratic
program (QP) safety filter that overrides a
potentially unsafe nominal control signal, to generate safe control actions.
High relative degree CBFs were considered in the articles \cite{2016Nguyen}, \cite{2019Xiao}, \cite{2021Xiao},  following the introduction of a non-overshooting control design in
\cite{KrsticBement2006} that is considered as the root of high relative degree CBF terminology. With the tools from \cite{KrsticBement2006},  mean-square stabilization of stochastic
nonlinear systems to an equilibrium at the barrier was solved in \cite{Li2020Mean}, and prescribed-time safety (PTSf) design which enforces
safety only for a finite time of interest to the user was proposed in \cite{Abel2022Prescribed}.

Most of the research in CBF safe control focuses on the systems described by ODEs while very few result is present for PDEs. The first result on safe PDE control is presented in \cite{Koga2023Safe}, extending the CBF-based safe control from ODEs to PDEs. By combining
high-relative-degree CBFs with PDE backstepping boundary control design, \cite{Koga2023Safe} proposed a safe control scheme for a Stefan model with actuator dynamics, whose mathematical model is an ODE-parabolic PDE-ODE system with a state-dependent moving boundary.  Furthermore, by designing an event-triggering mechanism, event-triggered safe control of such a PDE system is further proposed in \cite{Koga2023}. But the safe control design proposed in \cite{Koga2023Safe}, \cite{Koga2023} is not suitable for the control task considered in our paper, not only because of the big differences between the Stefan model in \cite{Koga2023Safe}, \cite{Koga2023} and the $2\times 2$ hyperbolic PDEs with in-domain unstable sources in our paper, but also due to the presence of model uncertainties.

In the presence of model uncertainties that often appear in practical applications, the guaranteed safety  based on the completely known model may fail. To truly enforce safety, it is necessary to
quantify safety in the context of unknown parameters, which motivates the study on adaptive
Control Barrier Functions (aCBFs) safe control, a topic that has drawn much attention in the safe control community  in the past three years. Most existing safety results with aCBFs are built upon the idea of
adaptive Control Lyapunov Functions (aCLFs) proposed in \cite{Krstic1995}.  The pioneering and representative work of aCBFs is \cite{Taylor2020}, which presents a variant of aCLFs in the context of safety: aCBFs, with which safety is adaptively achieved in the presence of parametric model uncertainty, rendering the system solutions constrained to a subset of the original safe set. An extension in \cite{2021Lopez} alleviates the
conservatism by using both parameter adaptation and data-driven model estimation, but conservatively restricts the set of initial conditions. Motivated by the above shortcomings, \cite{2021Maghenem} proposed a safety-critical adaptive control design that features a hybrid adaptive law that enables less conservative
behavior than the continuous update laws used before. Besides, \cite{2023Lopez} presented an adaptive safety framework that
permits the use of the certainty equivalence principle, simplifying the procedure in \cite{Taylor2020}, \cite{2021Lopez}.
More developments on aCBF safety with features of fixed-time adaptation \cite{2021Black} and multiple CBFs \cite{2021Isaly} were proposed as well. Unlike the above results, \cite{wang2023} takes into account the uncertainties in the control-input matrix. While the aforementioned works focus only on adaptive safety, \cite{2021Cohen} made an attempt to synthesize adaptive safety and exponential stabilization by unifying the proposed HO-RaCBF (High Order Robust
Adaptive Control Barrier Function), which inherits the robustness properties of zeroing CBFs, and ES-aCLF (Exponentially Stabilizing Adaptive Control Lyapunov
Function) as a tool to exponentially stabilize uncertain nonlinear systems, where the HO-RaCBF and ES-aCLF are built upon the aCBF paradigm
in \cite{Taylor2020} and the
classical aCLF formulation in \cite{Krstic1995}, respectively, with leveraging the concurrent learning technique \cite{Parikh2019Integral}.

The above aCBF safe control designs are for ODEs. To our best knowledge, it has not been pursued for PDEs. In this paper, we propose a new adaptive CBF method for not only an ODE, but also a system consisting of PDEs and ODEs, making use of the batch least-squares identifier (BaLSI) adaptive scheme that was first proposed in \cite{Karafyllis2019Adaptive,Karafyllis2018Adaptive} for nonlinear ODEs and then extended to PDEs in \cite{Karafyllis2019Adaptive1,JiAdaptive2021,Ji2021Adaptive}. Unlike most least-squares estimation methods \cite{Cao2014,Cao2016,Chowdhary2010,Wang2014,Krstic2009},  where persistency of excitation (PE) is required to guarantee
the parameter convergence, BaLSI does not require the PE assumption.
\subsection{Main Contribution}
Inspired by, but going beyond the application of UAV control in Sec. \ref{sec:UAV}, a generalized control problem of safe adaptive control of a cascade system in the configuration of ODE-heterodirectional coupled hyperbolic PDEs-ODE is dealt with in this paper, guaranteeing both safety and stability, even if uncertainties exist. The main contributions of this paper are:

1) This is the first result of safe adaptive PDE control utilizing aCBFs. The previous PDE adaptive control designs that focus on stabilization, such as \cite{Anfinsen2019Adaptive,Bernard2014,Bresch2014,Karafyllis2019Adaptive1,Smyshlyaev2007Adaptivea,Smyshlyaev2007Adaptiveb,Smyshlyaev2010Adaptive}, have not taken safety into account.

2) As compared to the representative work on aCBF-based safe control \cite{Taylor2020}, which constrains the system solution to a subset of the original safe set, we propose a new aCBF method that constrains the states in the original safe region after a finite time by leveraging the batch least-squares identifier. To be specific, \cite{Taylor2020} constrains the state to a set smaller than the original safe set, where, roughly speaking,
 the difference between the two sets is the product of the reciprocals of adaption gains and the square of the parameter estimate errors that are not ensured to be zero. In our design, even though the state is also constrained in a subset of the original safe set during the adaptive learning process that is finished in a finite time, after which the state will be maintained in the original safe set. Regarding the difference  from the original set in \cite{Taylor2020}, a larger adaption gain can possibly avoid constraining the state to a much smaller set. Even though we do not have such a property in the adaptive learning process because  the notion of "adaptation gain" is not directly present in BaLSI, we can quicken the parameter estimation process by adjusting design parameters to reduce the duration time of the state in the subset.

3) The result is new even if only the nominal control design is considered, i.e., without advancing to adaptive design under uncertainties. In other words, this is also the first safe boundary control design for hyperbolic PDEs, as compared to the safe boundary control design for parabolic PDEs in \cite{Koga2023Safe}.
\subsection{Notation}
\begin{itemize}
\item The symbol $\mathbb Z_+$ denotes the set of all non-negative integers, $\mathbb N$ denotes the set $\{1,2,\cdots\}$, i.e., the natural numbers without $0$, and $\mathbb R_+:=[0,+\infty)$.

\item Let $U\subseteq \mathbb R^n$ be a set with non-empty interior and let $\Omega\subseteq\mathbb R$
be a set. By $C^0 (U;\Omega)$, we denote the class of continuous
mappings on $U$, which takes values in $\Omega$. By $C^k (U;\Omega)$, where
$k \ge 1$, we denote the class of continuous functions on $U$,
which have continuous derivatives of order $k$ on $U$ and take
values in $\Omega$.

\item We use the notation $L^2(0, 1)$ for the standard space of the equivalence
class of square-integrable, measurable functions $f:(0,1)\to\mathbb R$, with $\|f\|^2:=\int_0^1 f(x)^2 dx<+\infty$ for $f \in L^2(0, 1)$. For an integer $k\ge 1$,
$H^{k}(0,1)$ denotes the Sobolev space of
functions in $L^2(0, 1)$ with all its weak derivatives up to order $k$ in $L^2(0, 1)$.

\item Let $u: \mathbb R_+\times [0,1]\rightarrow \mathbb R$  be given. We use the notation $u[t]$ to denote the profile of $u$ at certain $t\ge 0$, i.e., $u[t]=u(x,t)$ for all $x\in[0,1]$.

    \item The notation ${f^{(i)}}{(t)}$ denote $i$ times derivatives of $f$. We use ${\alpha_x^{(i)}}{(x,t)}$ to denote $i$ times derivatives with respect to $x$ of $\alpha(x,t)$. Similarly, ${\alpha_t^{(i)}}{(x,t)}$ denote $i$ times derivatives with respect to $t$ of $\alpha(x,t)$.
        \item Define $\underline x_j:= [x_1,x_2,\cdots,x_j]^T$, and $\underline \Gamma^{(i)}(t):= [\Gamma^{(1)}(t),\Gamma^{(2)}(t),\cdots,\Gamma^{(i)}(t)]^T$.
\end{itemize}

For ease of presentation,  we omit or simplify the arguments of functions and functionals when no confusion arises.
\section{Problem Formulation}
As mentioned in Sec. \ref{sec:UAV}, the considered plant  is
\begin{align}
\dot Y(t) &= AY(t) + Bw(0,t),\label{eq:o1}\\
{z_t}(x,t) &=  - q_1{z_x}(x,t)+{d_1}w(x,t),\label{eq:o2}\\
{w_t}(x,t) &= q_2{w_x}(x,t)+{d_2}z(x,t),\label{eq:o3}\\
z(0,t) &=pw(0,t),\label{eq:o4}\\
w(1,t) &=  x_1(t),\label{eq:o5}\\
\dot x_j(t) &= x_{j+1}(t)+f_j(\underline x_j),~j=1,\cdots,m-1,\label{eq:o6}\\
\dot x_m(t)&= f_m(\underline x_m)+\sum_{i=0}^{m-1} \bar q_iz^{(i)}(1,t)+M^TY(t)+ U(t)\label{eq:o7}
\end{align}
$\forall (x,t) \in [0,1]\times[0,\infty)$, where $X^T(t)=[x_1,x_2,\cdots,x_m] \in \mathbb{R}^{m}$ and $Y^{T}(t)=[y_1,y_2,\cdots,y_n]\in \mathbb{R}^{n}$ are ODE states, the scalars $z(x,t)\in \mathbb{R}, w(x,t)\in \mathbb{R}$ are states of the PDEs. The function $U(t)$ is the control input to be designed. The actuator $X$-ODE  is a strict-feedback nonlinear system, where the nonlinearities $f_j$ satisfy Assumption \ref{as:f}.
The $Y$-ODE \eqref{eq:o1} is a linear model, where the matrix $A$, the column vector $B$ are in the form of
\begin{align}
A=&\left(
  \begin{array}{cccccc}
    0 & 1 & 0 & 0 & \cdots & 0 \\
    0 & 0 & 1 & 0 & \cdots & 0 \\
     &  & \vdots &  &  &  \\
    0 & 0 & 0 & 0 & \cdots & 1 \\
    l_1 & l_2 & l_3 & \cdots & l_{n-1} & l_n \\
  \end{array}
\right), B=\left(
             \begin{array}{c}
               0 \\
               0 \\
               \vdots \\
               0 \\
               b \\
             \end{array}
           \right),\label{eq:ABC}
\end{align}
with arbitrary constants $l_1$, $l_2$, $l_3$, $\ldots$ , $l_{n-1}$, $l_n$, and $b>0$ (without any loss of generality for $b<0$). This indicates that the $Y$-ODE is in the controllable form which covers many practical models.
Other plant parameters in \eqref{eq:o2}--\eqref{eq:o7}, i.e., $d_1,d_2,p,q,\bar q_i, i=1,\cdots, m-1$, and $q_1>0,q_2>0$ that denote the transport speed, as well as the $n$-dimensional row vector $M^T$, are also arbitrary. The parameter $b$ that exists in the distal ODE and the coefficients $d_1,d_2$ of the PDE in-domain couplings that are the potentially destabilizing terms are unknown. The unknown parameters $d_1,d_2,b$  satisfy Assumption \ref{as:bound}.

\textbf{Control objective}: To exponentially regulate the overall system, including the plant $w$-PDE, $z$-PDE, $Y$-ODE and $X$-ODE,  and enforce "safety", defined here
as the \textbf{non-negativity} (without any loss of generality for non-positivity or not exceeding a nonzero setpoint) \textbf{of the output} of the distal ODE, i.e., the state furthest from the control input:
$$y_1(t)\ge0,~\forall t\in[0,\infty).$$

We conduct the control design based on such a sandwich PDE system \eqref{eq:o1}--\eqref{eq:o7}, not only because it covers the mathematical model of the UAV with cable hanging load mentioned in Sec. \ref{sec:UAV}, but also because the type of system is general since it comprises the coupled hyperbolic PDEs, linear and strict-feedback nonlinear ODEs, as well as the cascaded structures of PDE-ODE and ODE-PDE that appear in numerous PDE control articles. This safe control design is challenging since, between the control input and the state to be safely regulated,  there exist not only high-relative-degree nonlinear and linear ODEs  but also potentially unstable heterodirectional coupled hyperbolic PDEs. Besides, the unknown parameters associated with the unstable source in the PDE domain and input signals of the ODE to be kept safe make the task more difficult.
\begin{assm}\label{as:f}
The functions $f_j$ are $m-j$ times differentiable, and $f_j(\emph{\textbf{0}})=0$.
\end{assm}
\begin{assm}\label{as:bound}
The bounds of the unknown parameters $d_1,d_2,b$ are known and arbitrary, i.e.,
$\underline d_1\le d_1\le \overline d_1$, $\underline d_2\le d_2\le \overline d_2$, $0<\underline b\le b\le \overline b$.
\end{assm}
During the time interval $[0,\frac{1}{q_2}]$ no control action can reach the $Y(t)$ ODE due to the $w$-transport PDE. Therefore, we have to impose constraints on the initial states to ensure that the output $y_1(t)=C_1Y(t)$ stays in the safe region before the control action begins to regulate the $Y(t)$-ODE, where $C_i$ is a vector with $i$th entry as 1 and other entries are zero.
\begin{assm}\label{as:initial}
The initial values of the ODE and PDE states satisfy i) $y_1(0)\ge 0$; ii) $\Pi(\varsigma)= C_1e^{A\frac{\varsigma}{q_2}}Y(0)+C_1\frac{1}{q_2}e^{A\frac{\varsigma}{q_2}}\int_0^\varsigma e^{-A\frac{x}{q_2}}B(w(x,0) - \int_0^{x} F(x,y)z(y,0)dy
-  \int_0^{x} H(x,y)w(y,0)dy) dx\ge 0$ for $0< \varsigma< 1$ and $\Pi(1)> 0$, where the explicit expressions of $F(x,y)$ and $H(x,y)$ are given in \eqref{eq:F}, \eqref{eq:H}.
\end{assm}
In Assumption \ref{as:initial}, the condition i) implies there is no additional restriction on $y_1(0)$ but the original safe set; the condition ii)  is the sufficient and necessary condition of $y_1(t)\ge0$ on $t\in(0,\frac{1}{q_2})$ and $y_1(\frac{1}{q_2})>0$, i.e., before the control action  begins to regulate the $Y(t)$-ODE, which will be proven in Lemma \ref{lem:Y} latter.  Besides, the following assumption on the actuator state makes the control action for $Y$-ODE begin within the region of safe regulation.
\begin{assm}\label{as:initialx1}
The initial value of the actuator state $x_1(0)$ satisfies $x_1(0)>\int_0^1 {\Psi}(1,y)z(y,0)dy
+  \int_0^1 {\Phi}(1,y)w(y,0)dy +  \lambda (1){Y}(0)$, where explicit ${\Psi}(1,y)$, ${\Phi}(1,y)$, $\lambda (1)$  are given by \eqref{eq:kerf},  \eqref{eq:traker3}, \eqref{eq:exk1}, \eqref{eq:exk2}.
\end{assm}
\section{Nominal Safe Control design}\label{sec:TandB}
\subsection{First (ODE) nonundershooting backstepping transformation}\label{sec:tran1}
{Following \cite{KrsticBement2006}, we apply the transformation
\begin{align}
&z_i(t)=y_i(t)-g_{i-1}(\underline y_{i-1}(t)),~~i=1,\cdots,n\label{eq:zi}\\
&g_{0}=0,\label{eq:g0}\\
&g_{i}(\underline y_{i}(t))=-\kappa_iz_i(t)+\sum_{j=1}^{i-1}\frac{\partial g_{i-1}}{\partial{y_j}}y_{j+1},~i=1,\cdots,n-1\label{eq:gi}
\end{align}
where the positive design parameters $\kappa_i$, $i=1,\ldots, n$ are
to be determined later, to convert the distal $Y$-ODE \eqref{eq:o1} with \eqref{eq:ABC} into
\begin{align}
\dot Z(t)=A_{\rm z} Z(t)+Bw(0,t)-BK^TY(t)\label{eq:ZA}
\end{align}
where
\begin{align}
A_{\rm z}=\left(
  \begin{array}{ccccccc}
    -\kappa_1 & 1 & 0 & 0 & \cdots & 0& 0 \\
    0 & -\kappa_2 & 1 & 0 & \cdots & 0& 0 \\
     0 & 0 & -\kappa_3 & 1 & \cdots & 0& 0 \\
     &  & \vdots &  &  &  &\\
    0 & 0 & 0 & 0 & \cdots & -\kappa_{n-1} &1 \\
    0 & 0 & 0 & \cdots & 0  &0 &-\kappa_n \\
  \end{array}
\right)\label{eq:Az}
\end{align}
is Hurwitz, and the constant row vector
\begin{align}
&K^T=\frac{1}{b}\bigg[-l_1+\kappa_n\frac{\partial g_{n-1}}{\partial{y_{1}}},-l_2+\frac{\partial g_{n-1}}{\partial{y_{1}}}+\kappa_n\frac{\partial g_{n-1}}{\partial{y_{2}}},\ldots,\notag\\&-l_{n-1}+\frac{\partial g_{n-1}}{\partial{y_{n-2}}}+\kappa_n\frac{\partial g_{n-1}}{\partial{y_{n-1}}},-l_n+\frac{\partial g_{n-1}}{\partial{y_{n-1}}}-\kappa_n\bigg]_{1\times n}.\label{eq:K}
\end{align}
Note: if $n=1$, then $K$ in \eqref{eq:K} is equal to $-l_1-\kappa_1$, i.e., inserting $n=1$ into the last term in \eqref{eq:K} with recalling $g_0=0$. {Similarly, if $n=2$, then $K$ is a two-dimensional vector consisting of the first and last terms in \eqref{eq:K}.} Considering the linear system \eqref{eq:o1}, $g_{n-1}(\underline y_{n-1}(t))$ can be expressed as a linear combination of $y_{i}$, $i=1,\cdots,n-1$, i.e., $g_{n-1}(\underline y_{n-1}(t))=\sum_{i=1}^{n-1}\frac{\partial g_{n-1}}{\partial{y_{i}}}y_{i}$, where $\frac{\partial g_{n-1}}{\partial{y_{i}}}$ are constant. {This fact has been used in deriving \eqref{eq:K}.}

The transformed states $Z(t)=[z_1,\ldots,z_n]^T$ are indeed high-relative-degree ODE CBFs since,  they need to be kept non-negative on $t\in[\frac{1}{q_2},\infty)$ by the control design for
achieving the original safety goal. Please note that the control design guarantees the CBFs $z_i$ non-negative from $t=\frac{1}{q_2}$ instead of $t=0$ because the control action needs to take the time $\frac{1}{q_2}$ to pass through the first-order hyperbolic PDEs and begins to regulate the distal ODE until $t=\frac{1}{q_2}$.}
\subsection{Second PDE backstepping transformation}
In order to remove the in-domain coupling  destabilizing terms from  the $2\times 2$ hyperbolic PDE system and compensate the term  $BK^TY(t)$ in \eqref{eq:ZA}, we introduce the following backstepping transformation \cite{Meglio2017Stabilization}:
\begin{align}
\alpha (x,t) =&z(x,t) - \int_0^x {\phi}(x,y)z(y,t)dy\notag\\
& -  \int_0^x {\varphi}(x,y)w(y,t)dy -  \gamma (x){Y}(t),\label{eq:contran1a}\\
\beta (x,t) =& w(x,t) - \int_0^x {\Psi}(x,y)z(y,t)dy \notag\\
&-  \int_0^x {\Phi}(x,y)w(y,t)dy -  \lambda (x){Y}(t)\label{eq:contran1b}
\end{align}
where ${\phi},\varphi,\gamma,\Psi,\Phi,\lambda$ are defined in Appendix\ref{sec:ker}. We named $\beta$-PDE as "PDE CBF", because, as is typical with CBFs, the function $\beta(x,t)$ should be ensured non-negative (for $t\in[\frac{1}{q_2},\infty)\times x\in[0,1]$) in pursuing the original safety goal, which will be seen clear later.

By \eqref{eq:contran1a}, \eqref{eq:contran1b}, the system \eqref{eq:o2}--\eqref{eq:o5} with \eqref{eq:ZA} is converted into
\begin{align}
\dot Z (t) =& A_{\rm z}Z (t) + {B}\beta (0,t), \label{eq:targ5}\\
  \alpha (0,t) =& p\beta  (0,t),\label{eq:targ2}\\
{\alpha _t}(x,t) =&  - {q_1}{\alpha _x}(x,t),\label{eq:targ1}\\
{\beta _t}(x,t) =& {q_2}{\beta _x}(x,t),\label{eq:targ4}\\
\beta(1,t) =& x_1(t)-\Gamma(t)\label{eq:targ8}
\end{align}
where the function $\Gamma(t)$ is
\begin{align}
\Gamma(t)=&\int_0^1 {\Psi}(1,y)z(y,t)dy +  \int_0^1 {\Phi}(1,y)w(y,t)dy\notag\\
& + \lambda (1){Y}(t).\label{eq:gamma}
\end{align}
\subsection{Third (ODE) nonundershooting  backstepping transformation}\label{sec:ODEcontroldesign}
Similar to Sec. \ref{sec:tran1}, we introduce the following modified backstepping transformations for proximal $X$-ODE,
\begin{align}
&h_i(t)=x_i(t)-\tau_{i-1}-\Gamma^{(i-1)}(t),~~i=1,\cdots,m\label{eq:ncti}\\
&\tau_{0}=0,\label{eq:tau0}\\
&\tau_{i}(\underline x_{i}(t),\underline \Gamma^{(i-1)}(t))=-c_ih_i(t)-f_i(\underline x_{i}(t))\notag\\&+\sum_{j=1}^{i-1}\bigg[\frac{\partial\tau_{i-1}}{\partial{x_j}}(x_{j+1}+f_j(\underline x_{j}(t)))+\frac{\partial\tau_{i-1}}{\partial{\Gamma^{(j-1)}(t)}}\Gamma^{(j)}(t)\bigg],\notag\\&i=1,\cdots,m-1\label{eq:taui}
\end{align}
where the transformed states $h_i(t)$ of the proximal ODE are also high-relative-degree CBFs since, in addition to the CBFs $z_i(t)$ in \eqref{eq:zi} and the PDE CBF $\beta(x,t)$ in \eqref{eq:contran1b}, $h_i(t)$ need to be kept non-negative for all time for achieving the original safety goal. The positive constants $c_1,\ldots,c_m$ are design parameters to be determined later for ensuring both safety and
stability.
Recalling \eqref{eq:gamma}, the function $\Gamma^{(i)}(t)$ is obtained as
\begin{align}
&\Gamma^{(i)}(t)
=-\sum_{j=0}^{j=i-1}q_1R_{i-1-j}(1)z_{t}^{(j)}(1,t)\notag\\&+\sum_{j=0}^{j=i-1}q_1R_{i-1-j}(0)z_{t}^{(j)}(0,t)+\int_0^1 R_i(y)z(y,t)dy\notag\\
&+\sum_{j=0}^{j=i-1}q_2P_{i-1-j}(1)w_{t}^{(j)}(1,t)\notag\\&-\sum_{j=0}^{j=i-1}\left(q_2P_{i-1-j}(0)-\lambda (1)A^{i-1-j}B\right)w_{t}^{(j)}(0,t)\notag\\&+\int_0^1 P_i(y)w(y,t)dy+\lambda (1) A^{i}{Y}(t),~i=1,\cdots,m\label{eq:nctG}
\end{align}
where the functions $R_i$, $P_i$ are defined by
\begin{align}
R_i(y)&=q_1R_{i-1}'(y)+d_2P_{i-1}(y),~i=1,\cdots,m\label{eq:Ri}\\
P_i(y)&=-q_2P_{i-1}'(y)+d_1R_{i-1}(y),~i=1,\cdots,m\\
R_0(y)&={\Psi}(1,y),~
P_0(y)={\Phi}(1,y),\label{eq:P0}
\end{align}
with the explicit solutions of ${\Psi}(1,y)$ and ${\Phi}(1,y)$ given by  \eqref{eq:exk1}, \eqref{eq:exk2} and \eqref{eq:traker3}, \eqref{eq:F}, \eqref{eq:H}. Please note that the $m$ order derivatives of ${\Psi}(1,y)$ and ${\Phi}(1,y)$ exist according to Theorem 5 in \cite{Vazquez2011Backstepping}.

Through the transformations \eqref{eq:zi}--\eqref{eq:gi}, \eqref{eq:contran1a}, \eqref{eq:contran1b}, \eqref{eq:ncti}--\eqref{eq:nctG}, now we convert the original system \eqref{eq:o1}--\eqref{eq:o7} to the following target system
\begin{align}
\dot Z(t) &= A_{\rm Z}Z(t) + B\beta (0,t),\label{eq:target1}\\
{\alpha _t}(x,t) &=  - q_1{\alpha _x}(x,t),\label{eq:target2}\\
{\beta _t}(x,t) &= q_2{\beta _x}(x,t),\label{eq:target3}\\
\alpha (0,t) &= p{\beta}(0,t),\label{eq:target4}\\
\beta(1,t)&=h_1(t),\label{eq:target5}\\
\dot h_i(t)&=-c_i h_i(t)+h_{i+1}(t), i=1,\cdots, m-1\label{eq:targetym-1}\\
\dot h_m(t)&=-c_m h_m(t)\label{eq:targetym}
\end{align}
by choosing the control input as
\begin{align}
U(t)=&\tau_m-{\sum_{i=0}^{m-1} \bar q_iz_t^{(i)}(1,t)-M^TY(t)}+\Gamma^{(m)}(t)\notag\\:= &\mathcal U(\chi(t);\theta)\label{eq:U1}
\end{align}
where $\chi(t)$, which denotes all the system signals used in the control law, can be written as
\begin{align}
\chi(t)=&\bigg[x_1(t),\cdots, x_{m}(t),w(0,t),\cdots ,w_t^{(m-1)}(0,t),\notag\\& z(1,t),\cdots ,z_t^{(m-1)}(1,t),\int_0^1 \mathcal R (x) w(x,t) dx,\notag\\&  \int_0^1 \mathcal P (x)z(x,t)dx, y_1(t),\cdots,y_n(t)\bigg],\label{eq:chi}
\end{align}
for some $\mathcal R (x), \mathcal P (x)$ consisting of $R_i(x),P_i(x)$ defined in \eqref{eq:Ri}--\eqref{eq:P0}.
Writing
\begin{align}\theta=[d_1,d_2,b]^T\label{eq:theta}\end{align}
after $``;"$ in $\mathcal U(\chi(t);\theta)$  emphasizes the fact that the control law depends on the unknown parameters $d_1,d_2,b$. The control law is explicit considering the explicit solutions of ${\Psi}(1,y)$ and ${\Phi}(1,y)$ are obtained in \eqref{eq:exk1}, \eqref{eq:exk2} and \eqref{eq:traker3}, \eqref{eq:F}, \eqref{eq:H}.

The calculation details in the third transformation is shown in Appendix\ref{sec:odeback}.
\subsection{Selection of nonundershooting design parameters}\label{sec:designparameter}
{We choose the design parameters $\kappa_1,\ldots,\kappa_{n-1}$ ($\kappa_n>0$ is  free) to satisfy
\begin{align}
\kappa_{i}&>
\frac{(\sum_{j=1}^{i-1}\frac{\partial g_{i-1}}{\partial{y_j}}C_{j+1}-C_{i+1})Y(\frac{1}{q_2})}{C_{i}Y\left(\frac{1}{q_2}\right)-g_{i-1}\left(\underline C_{i-1}Y\left(\frac{1}{q_2}\right)\right)}:=\check \kappa_i(b)\label{eq:kappai}
\end{align}
for $i=1,\cdots, n-1$, which includes the unknown parameter $b$, where $Y(\frac{1}{q_2})$ is expressed as the initial values of $z$-PDE, $w$-PDE, and $Y$-ODE as
\begin{align*}
Y(\frac{1}{q_2})=&e^{A\frac{1}{q_2}}Y(0)+\int_0^{\frac{1}{q_2}} e^{A({\frac{1}{q_2}}-\tau)}B\bigg(w(q_2\tau,0)\notag\\
 &- \int_0^{q_2\tau} F(q_2\tau,y)z(y,0)dy\notag\\& -  \int_0^{q_2\tau} H(q_2\tau,y)w(y,0)dy\bigg) d\tau,
\end{align*}
which will be seen clearly in the proof of Lemma \ref{lem:Y}. The expressions of $F(x,y)$ and $H(x,y)$ are given in \eqref{eq:F}, \eqref{eq:H}.  The purpose of the selection of the design parameters  \eqref{eq:kappai} is to make CBFs $z_i(t)$, $i=2,\cdots,n$, positive at the time $t=\frac{1}{q_2}$ when the control action reaches the distal ODE, using initial values of the plant states. This will be clear in the proof of Lemma  \ref{cl:zi0}.}

The design parameters $c_1,\ldots,c_m$ are selected as
\begin{align}
c_i>&\max\{2, \check c_i(\theta)\},~~i=1,\cdots,m-1,~~c_m>1\label{eq:ci}
\end{align}
where
\begin{align}
&\check c_i(\theta)=\frac{1}{x_i(0)-\tau_{i-1}(\underline x_{i-1}(0),\underline \Gamma^{(i-2)}(0))-\Gamma^{(i-1)}(0)}\notag\\&\times\bigg[-x_{i+1}(0)-f_i(\underline x_i(0))\notag\\&+\sum_{j=1}^{i-1}\bigg[\frac{\partial\tau_{i-1}(\underline x_{i-1}(0),\underline \Gamma^{(i-2)}(0))}{\partial{x_j}(0)}(x_{j+1}(0)+f_j(\underline x_j(0))\notag\\&+\frac{\partial\tau_{i-1}(\underline x_{i-1}(0),\underline \Gamma^{(i-2)}(0))}{\partial{\Gamma^{(j-1)}(0)}}\Gamma^{(j)}(0)\bigg]+\Gamma^{(i)}(0)\bigg],\label{eq:underc}
\end{align}
for  $i=1,\cdots,m-1$, which includes the unknown parameters $\theta=[d_1,d_2,b]^T$.
The purpose of the choices of design parameters \eqref{eq:ci} is to make CBFs $h_i(t)$, $i=2,\cdots,m$, positive at $t=0$ and ensure the stability through the Lyapunov analysis, which will be clear later in the proofs of Lemma \ref{cl:hi0} and Theorem \ref{th:th1}, respectively. Please note that the denominators in \eqref{eq:underc} and \eqref{eq:kappai} are nonzero, which will be clear later in the proof of Lemmas \ref{cl:zi0} and \ref{cl:hi0}.

Summary of the nominal safe control design: the proposed nominal safe control design for sandwich hyperbolic PDE systems includes the following two steps. 1) Applying the backstepping transformations \eqref{eq:zi}--\eqref{eq:gi}, \eqref{eq:contran1a}, \eqref{eq:contran1b}, \eqref{eq:ncti}--\eqref{eq:taui}  to convert the original plant into the target system \eqref{eq:target1}--\eqref{eq:targetym} that is not only exponentially stable but also has safe property in the sense that the states, i.e., the CBFs,  are non-negative provided that their initial values (the values at the first time when the control action comes into play) are non-negative. 2) Choose the design parameters $c_i$,  $i=1,\ldots,m-1$ and $\kappa_i$,  $i=1,\ldots,n-1$ to initialize positively the CBFs $h_i(t)$, $z_i(t)$.
\subsection{Result with nominal safe control}
For the time interval $[0,\frac{1}{q_2}]$, no control action reaches
the distal $Y (t)$-ODE, whose safety is ensured under the given initial conditions, which is shown in the following lemma.
\begin{lema}\label{lem:Y}
For the time period no control action reaches the $Y$-ODE, $y_1(t)$ is kept in the safe region, i.e.,
\begin{align}
y_1(t)\ge 0,~\forall t\in\left [0,\frac{1}{q_2}\right)~and~~y_1\left (\frac{1}{q_2}\right)> 0\label{eq:y1q2}
\end{align}
under Assumption \ref{as:initial} regarding the initial data.
\end{lema}
\begin{proof}
Defining
\begin{align}
\eta (x,t) = &w(x,t) - \int_0^x F(x,y)z(y,t)dy\notag\\& -  \int_0^x H(x,y)w(y,t)dy\label{eq:eta}
\end{align}
where $F(x,y),H(x,y)$ are given by \eqref{eq:ex5}--\eqref{eq:ex8},
recalling \eqref{eq:o1}--\eqref{eq:o4}, we have
\begin{align}
\dot Y(t) &= AY(t) + B\eta (0,t),\\
{\eta_t}(x,t) &= q_2{\eta_x}(x,t).
\end{align}
We do not display the right boundary condition of $\eta$ here because it is not required in analyzing the evolution of $y_1(t)$
on $t\in[0,\frac{1}{q_2}]$, which only depends on the initial values of $Y,z,w$. By recalling \eqref{eq:eta}, the solution of $Y(t)$ is obtained as
\begin{align}
Y(t)=&e^{At}Y(0)+\int_0^t e^{A(t-\tau)}B\bigg(w(q_2\tau,0) \notag\\&- \int_0^{q_2\tau} F(q_2\tau,y)z(y,0)dy \notag\\
&-  \int_0^{q_2\tau} H(q_2\tau,y)w(y,0)dy\bigg) d\tau,~0\le t\le \frac{1}{q_2}.\label{eq:Yuncon}
\end{align}
Recalling Assumption \ref{as:initial} with letting $x=q_2\tau$ and $\varsigma=q_2t$, we know from \eqref{eq:Yuncon} that
$y_1(t)=C_1Y(t)\ge 0,~~t\in[0,\frac{1}{q_2})$ and $y_1(\frac{1}{q_2})>0$.
\end{proof}
Next, we present two lemmas regarding initializing positively the CBFs $z_i(t)$, $h_i(t)$ of distal and proximal ODEs, at $t=\frac{1}{q_2}$ and $t=0$,  by the selection of the design parameters $\kappa_i$ and $c_i$ in Sec. \ref{sec:designparameter}, respectively.
{\begin{lema}\label{cl:zi0}
The positivity of $z_i(\frac{1}{q_2})$, i.e., the values of high-relative-degree ODE CBFs $z_i$ at the time instant when the control action reaches the distal ODE, is ensured, i.e., $z_i(\frac{1}{q_2})>0$, $i=1,\cdots,n$, under the design parameters $\kappa_i$, $i=1,\cdots,n-1$ satisfying \eqref{eq:kappai}.
\end{lema}
\begin{proof}
We prove $z_{i}(\frac{1}{q_2})>0$, $i=1,\cdots,n$ by the induction method as well.
Recalling \eqref{eq:zi}, \eqref{eq:g0}, and \eqref{eq:y1q2},  we know the base case
\begin{align}
z_1\left(\frac{1}{q_2}\right)=y_1 \left(\frac{1}{q_2}\right)>0.\label{eq:z1less0}
\end{align}
We then show the induction step: if $z_i(\frac{1}{q_2})>0$,  then $z_{i+1}(\frac{1}{q_2})>0$ under the choices of  $\kappa_i$ in \eqref{eq:kappai}. According to \eqref{eq:zi}, \eqref{eq:gi}, we know
\begin{align}\label{eq:zi10}
&z_{i+1}\left(\frac{1}{q_2}\right)
=y_{i+1}\left(\frac{1}{q_2}\right)+k_{i}z_{i}\left(\frac{1}{q_2}\right)-\sum_{j=1}^{i-1}\frac{\partial g_{i-1}}{\partial{y_j}}y_{j+1}\left(\frac{1}{q_2}\right)\notag\\
&=\bigg[C_{i+1}-\sum_{j=1}^{i-1}\frac{\partial g_{i-1}}{\partial{y_j}}C_{j+1}\bigg]Y\left(\frac{1}{q_2}\right)\notag\\&+\kappa_i\bigg[C_{i}Y\left(\frac{1}{q_2}\right)-g_{i-1}\left(\underline C_{i-1}Y\left(\frac{1}{q_2}\right)\right)\bigg].
\end{align}
Inserting \eqref{eq:kappai} where the denominator $C_{i}Y\left(\frac{1}{q_2}\right)-g_{i-1}\left(\underline C_{i-1}Y\left(\frac{1}{q_2}\right)\right)=y_i(\frac{1}{q_2})-g_{i-1}(\underline y_{i-1}(\frac{1}{q_2}))=z_i(\frac{1}{q_2})>0$ which is obtained from the inductive hypothesis $z_i(\frac{1}{q_2})>0$, we have from \eqref{eq:zi10} that $z_{i+1}(\frac{1}{q_2})>0$. Therefore, recalling the base case $z_1(\frac{1}{q_2})>0$ in \eqref{eq:z1less0}, together with the induction step proved above, we immediately conclude that $z_{i+1}(\frac{1}{q_2})>0$ for $i=1,\cdots,n-1$ under the choices of  $\kappa_i$ in \eqref{eq:kappai} where the denominators are nonzero (positive).
\end{proof}}
\begin{lema}\label{cl:hi0}
The high-relative-degree ODE CBFs $h_i$ are initialized positively, i.e., $h_i(0)>0$, $i=1,\cdots,m$, under the design parameters $c_i$, $i=1,\cdots,m-1$ satisfying \eqref{eq:ci}.
\end{lema}
\begin{proof}
We prove $h_{i}(0)>0$, $i=1,\cdots,m$ by the induction method.
Recalling \eqref{eq:o5}, \eqref{eq:contran1b},  Assumption \ref{as:initialx1}, and \eqref{eq:target5}, we know the base case
\begin{align}
h_1(0)=\beta (1,0)>0.\label{eq:h1less0}
\end{align}
We then show the induction step: if $h_i(0)>0$,  then $h_{i+1}(0)>0$ under the choices of  $c_i$ in \eqref{eq:ci}. According to \eqref{eq:ncti}, \eqref{eq:taui}, we know
\begin{align}\label{eq:hi10}
&h_{i+1}(0)=x_{i+1}(0)-\tau_{i}(0)-\Gamma^{(i)}(0)\notag\\
&=x_{i+1}(0)+c_ih_i(0)+f_i(\underline x_i(0))\notag\\&-\sum_{j=1}^{i-1}\bigg[\frac{\partial\tau_{i-1}(\underline x_{i-1}(0),\underline \Gamma^{(i-2)}(0))}{\partial{x_j}(0)}(x_{j+1}(0)+f_j(\underline x_j(0))\notag\\&+\frac{\partial\tau_{i-1}(\underline x_{i-1}(0),\underline \Gamma^{(i-2)}(0))}{\partial{\Gamma^{(j-1)}(0)}}\Gamma^{(j)}(0)\bigg]-\Gamma^{(i)}(0)
\end{align}
for $i=1,\cdots, m-1$. Inserting \eqref{eq:ci} with \eqref{eq:underc} where $x_i(0)-\tau_{i-1}(0)-\Gamma^{(i-1)}(0)=h_i(0)>0$ which is obtained from the inductive hypothesis $h_i(0)>0$, we have from \eqref{eq:hi10} that $h_{i+1}(0)>0$. Therefore, recalling the base case $h_1(0)>0$ in \eqref{eq:h1less0}, together with the induction step proved above, we immediately conclude that $h_{i+1}(0)>0$ for $i=1,\cdots,m-1$ under the choices of  $c_i$ in \eqref{eq:ci} where the denominators are nonzero (positive).
\end{proof}
The above three lemmas will be used to prove the safety property stated in the following theorem.
\begin{thme}\label{th:th1}
For initial data $w[0]\in C^{m-1}([0,1])$, $z[0]\in C^{m-1}([0,1])$, $X(0)\in \mathbb R^m$, $Y(0)\in \mathbb R^n$ satisfying Assumptions \ref{as:f}, \ref{as:initial}, \ref{as:initialx1}, for design parameters $c_i$, $i=1,\ldots,m$ satisfying \eqref{eq:ci} and $\kappa_i$, $i=1,\ldots,n-1$ satisfying \eqref{eq:kappai}, the closed-loop system including the plant \eqref{eq:o1}--\eqref{eq:o7} with the nominal safe (output-positive) controller \eqref{eq:U1} has the following properties:

1) Safety (output positivity) is ensured in the sense that
$y_1(t)\ge 0,~\forall t\ge 0$.

2) Exponential regulation is achieved in the sense that $\|w(\cdot,t)\|^2+\|z(\cdot,t)\|^2+|X(t)|^2+|Y(t)|^2$ is exponentially convergent to zero.

3) The control input is exponentially convergent to zero, i.e., $\lim_{t\to\infty}U(t)=0$.
\end{thme}
\begin{proof}
1)
We prove the safety in property 1 by showing that the non-negativity of the CBFs is guaranteed: $h_i(t)\ge 0$, $i=1,\cdots,m$, on $t\in[0,\infty)$, $\beta(x,t)\ge 0$ on $x\in[0,1]\times t\in[\frac{1}{q_2},\infty)$, and $z_i(t)\ge 0$, $i=1,\cdots,n$, on $t\in[\frac{1}{q_2},\infty)$.

According to the structure \eqref{eq:targetym-1}, \eqref{eq:targetym} and Lemma \ref{cl:hi0}, we obtain the positivity of the ODE CBF, i.e., $h_{i}(t)\ge 0$, $i=1,\cdots,m$, for all $t\ge 0$.

The solution of the transport PDE \eqref{eq:target3}, \eqref{eq:target5}  is
\begin{align}\label{eq:beta}
\beta(x,t)=\begin{cases} \beta(x+q_2t,0)~~~~~0\le x+q_2t< 1,\\ h_1(t-(1-x)/q_2)~~~ x+q_2t\ge 1,\end{cases}
\end{align}
for all $t\in[0,\infty),x\in[0,1]$.
Recalling $h_{1}(t)\ge 0$ for all $t\ge 0$ obtained above, it follows from \eqref{eq:beta} that the non-negativity of the PDE CBF, i.e., $\beta(x,t)\ge 0$ for all $t\ge \frac{1}{q_2}$, $x\in[0,1]$. It implies
\begin{align}\label{eq:beta0}
\beta(0,t)\ge 0,~~~\forall t\ge \frac{1}{q_2}.
\end{align}
It follows from Lemma \ref{cl:zi0}, \eqref{eq:target1}, \eqref{eq:Az},  and \eqref{eq:beta0} that $z_i(t)\ge 0$, $i=1,\cdots,n$, on $t\in[\frac{1}{q_2},\infty)$. As a result, $y_1(t)=z_1(t)\ge 0$ for $t\in[\frac{1}{q_2},\infty)$. Recalling Lemma \ref{lem:Y},
the property 1 is obtained.

{2) We show the exponential regulation in property 2 is achieved by Lyapunov analysis. Consider the Lyapunov function,
\begin{align}\label{eq:V}
&V(t)=Z(t)^TPZ(t)+\frac{r}{2}\sum_{i=1}^{m}h_i(t)^2\notag\\&+\frac{1}{2}\int_0^1 {{}{e^{ - x}}{\alpha}{{(x,t)}^2}} dx+\frac{1}{2}\int_0^1 {{a_{0}}{e^{ x}}{\beta}{{(x,t)}^2}} dx
\end{align}
where the positive definite matrix $P= {P}^T$ is the solution to the
Lyapunov equation $A_{\rm Z}^TP+PA_{\rm Z}=-Q$ for some $Q={Q}^T>0$, and where the positive constants $r,a_0$ are to be determined later.
Defining
\begin{align}
\Xi(t)=\|\beta(\cdot,t)\|^2+\|\alpha(\cdot,t)\|^2+\sum_{i=1}^{m}h_i(t)^2+|Z(t)|^2\label{eq:Xi}
\end{align}
we have
\begin{align}
\xi_1\Xi(t)\le V(t)\le \xi_2\Xi(t)\label{eq:Xibound}
\end{align}
for some positive $\xi_1$ and $\xi_2$.
Taking the derivative of \eqref{eq:V} along \eqref{eq:target1}--\eqref{eq:targetym}, applying Young's inequality, we get
 \begin{align}
&\dot V(t)\le - \frac{1}{2}\lambda_{\rm min}(Q){ {
 Z(t)} ^2} \notag\\
& - \bigg(\frac{{q_2} {a_0}}{2} -\frac{{q_1q^2} }{2}- \frac{{2|PB|^2}}{\lambda_{\rm min}(Q)}\bigg) \beta {(0,t)^2}\notag\\
& -rc_1 h_1(t)^2+rh_1h_2-rc_2 h_2(t)^2+rh_2h_3(t)+\cdots\notag\\&-rc_{m-1}h_{m-1}(t)^2+rh_{m-1}(t)h_{m}(t)-rc_{m}h_{m}(t)^2\notag\\
&-\frac{1}{2}q_1{}{e^{ - 1}}{\alpha}{{(1,t)}}^2-\frac{1}{2}q_1\int_0^1 {{}{e^{ - x}}{\alpha}{{(x,t)}}^2} dx\notag\\&+\frac{{a_{0}q_2}{e}}{2}h_1(t)^2-\frac{{a_{0}}q_2}{2}\int_0^1 {e^{ x}}{\beta}{{(x,t)}}^2 dx.
\label{eq:dLy}
 \end{align}
Choosing the analysis parameters ${a_0}$ as
\begin{align}
{a_0}\ge{\frac{q_1q^2}{q_2}}+\frac{{4|PB|^2}}{{q_2}\lambda_{\rm min}(Q)}\label{eq:ra}
\end{align}
applying Young's inequality, one obtains
 \begin{align}
&\dot V(t) \le-\left(r\left(c_1-\frac{1}{2}\right)-\frac{1}{2}{q_2}{a_0}{e} \right)h_1(t)^2\notag\\&-\sum_{i=2}^{m-1}r(c_i-1)h_i(t)^2-r\left(c_m-\frac{1}{2}\right)h_m(t)^2- \frac{\lambda_{\rm min}(Q)}{2}{ {
 Z(t)} ^2}\notag\\
&-\frac{1}{2}q_1{}{e^{ - 1}}{\alpha}{{(1,t)}}^2-\frac{1}{2}q_1\int_0^1 {{}{e^{ - x}}{\alpha}{{(x,t)}}^2} dx\notag\\&-\frac{{a_{0}}q_2}{2}\int_0^1 {e^{ x}}{\beta}{{(x,t)}}^2 dx
 \end{align}
for $t\ge 0$.
Recalling the conditions of the design parameters $c_i$ in \eqref{eq:ci}, and choosing the analysis parameter $r$ as
 \begin{align}
 r&>\frac{1}{3}{q_2}{a_0}{e}+1,\label{eq:r}
 \end{align}
we obtain
\begin{align}
\dot V\le& -\sum_{i=1}^{m}h_i(t)^2- \frac{1}{2}\lambda_{\rm min}(Q){ {
 |Z(t)|} ^2}\notag\\&-\frac{{e^{ - 1}q_1}}{2}\int_0^1 {{\alpha}{{(x,t)}}^2} dx-\frac{1}{2}q_2{{a_{0}}}\int_0^1 {{\beta}{{(x,t)}}^2} dx\notag\\
\le& -\sigma_0 V(t)
\end{align}
where
\begin{align}
\sigma_0=\frac{1}{\xi_2}\min\left\{1,\frac{1}{2}\lambda_{\rm min}(Q),\frac{1}{2}{q_1}e^{ - 1}, \frac{1}{2}q_2{{a_{0}}} \right\}.\label{eq:sigma0}
\end{align}
It follows that
\begin{align}
\Xi(t)\le \frac{\xi_2}{\xi_1}\Xi(0)e^{-\sigma_0 t}\label{eq:targetstability}
\end{align}recalling \eqref{eq:Xibound}.
Next, we show the exponential regulation of the original states from the target system's exponential stability obtained \eqref{eq:targetstability}.
The  inverse of \eqref{eq:contran1a}, \eqref{eq:contran1b} is obtained as
\begin{align}
z (x,t) =&\alpha(x,t) - \int_0^x {\bar\phi}(x,y)\alpha(y,t)dy\notag\\
& -  \int_0^x {\bar\varphi}(x,y)\beta(y,t)dy -  \bar\gamma (x){Y}(t),\label{eq:Icontran1a}\\
w (x,t) =& \beta(x,t) - \int_0^x {\bar\Psi}(x,y)\alpha(y,t)dy \notag\\
&-  \int_0^x {\bar\Phi}(x,y)\beta(y,t)dy -  \bar\lambda (x){Y}(t)\label{eq:Icontran1b}
\end{align}
where the kernels ${\bar\phi},\bar\varphi,\bar\gamma,\bar\Psi,\bar\Phi,\bar\lambda$ are given in Appendix\ref{sec:inker}.
The inverse of the transformation \eqref{eq:zi}--\eqref{eq:gi} is
\begin{align}
Y(t)=T_{\rm Z}^{-1}Z(t)\label{eq:ZTY}
\end{align}
where $T_{\rm Z}$ is a lower-triangular constant matrix.

According to the exponential convergence to zero of $|Z(t)|$, $\|\beta(\cdot,t)\|^2$, $\|\alpha(\cdot,t)\|^2$,  we obtain the exponential convergence to zero of $|Y(t)|$ recalling \eqref{eq:ZTY}, and the exponential convergence to zero of $\|w(\cdot,t)\|^2+\|z(\cdot,t)\|^2$ by \eqref{eq:Icontran1a}, \eqref{eq:Icontran1b}.
According to \eqref{eq:target5}--\eqref{eq:targetym}, together with the exponential convergence to zero of $\sum_{i=1}^{m}h_i(t)^2$, we have that $\sum_{i=0}^{m-1} \beta_t^{(i)}(1,t)^2$ are exponentially convergent to zero. Applying the method of characteristics, recalling \eqref{eq:target2}--\eqref{eq:target5}, we have that  $\sum_{i=0}^{m-1} \beta_t^{(i)}(0,t)^2$, $\sum_{i=0}^{m-1} \alpha_t^{(i)}(0,t)^2$, and $\sum_{i=0}^{m-1} \alpha_t^{(i)}(1,t)^2$ are exponentially convergent to zero.
Then taking the time derivatives of \eqref{eq:Icontran1a}, \eqref{eq:Icontran1b} at $x=0$ and $x=1$ along the target system \eqref{eq:target1}--\eqref{eq:targetym}, implementing Cauchy-Schwarz inequality, recalling the exponential convergence results obtained above, we then have the exponential convergence to zero of the signals $\sum_{i=0}^{m-1} w_t^{(i)}(0,t)^2$, $\sum_{i=0}^{m-1} w_t^{(i)}(1,t)^2$, $\sum_{i=0}^{m-1} z_t^{(i)}(0,t)^2$, $\sum_{i=0}^{m-1} z_t^{(i)}(1,t)^2$.
Recalling Assumption \ref{as:f} and \eqref{eq:ncti}--\eqref{eq:nctG} (step by step from $i=1$ to $i=m$), we obtain that $|X(t)|$ is exponentially convergent to zero as well. We thus obtain property 2.

3) The exponential convergence of the control input \eqref{eq:U1} to zero in property 3 is straightforward to obtain, because the exponential convergence of all signals in the controller \eqref{eq:U1} have been guaranteed by property 2 (the exponential convergence to zero of $\sum_{i=0}^{m-1} w_t^{(i)}(0,t)^2$, $\sum_{i=0}^{m-1} z_t^{(i)}(0,t)^2$, $\sum_{i=0}^{m-1} w_t^{(i)}(1,t)^2$, $\sum_{i=0}^{m-1} z_t^{(i)}(1,t)^2$ are obtained at the end of the proof of property 2. The property 3 is thus obtained.}
\end{proof}
The nominal control design in this section is the first safe control design for hyperbolic PDEs, where both safety and exponential stability are guaranteed. Even though the safety goal (positivity) is not imposed on the original PDE states $z(x,t),w(x,t)$, the PDE states need to be incorporated in the safe control design to guarantee the safety (positivity) of the distal ODE, where the PDE CBF $\beta(x,t)$ is introduced and whose non-negativity on the entire space domain and the time domain (from the time the control action reaches the distal ODE to infinity) is to be ensured.

Next, considering uncertainties that often exist in practical applications, we propose a new aCBF method and present a safe adaptive controller $U_a(t)$ that guarantees both safety and exponential regulation in the presence of the uncertainties that exist in both PDEs and the ODE at the uncontrolled PDE boundary.
\section{Safe adaptive control design}\label{sec:adaptive}
\subsection{Adaptive controller}
\subsubsection{Certainty equivalence control law}
First, we build a certainty equivalence adaptive controller, which is potentially unsafe, by replacing the unknown parameters $\theta$ in the nominal control input $U$ with the parameter estimate $\hat\theta$, i.e.,
\begin{align}
U_d(t):=  \mathcal U(\chi(t);\hat\theta(t_i)),~~~t\in[t_i,t_{i+1})\label{eq:Ud}
\end{align}
where $\hat\theta(t_i)=[\hat d_1,\hat d_2,\hat b]^T$, an estimate generated with a triggered batch least-squares identifier that is updated along a sequence of time instants $t_i$ and uses the plant states in a time interval before $t_i$ to produce the parameter estimates, will be defined in Sec. \ref{sec:ls}, and where the sequence of triggering time instants $\{t_i\ge0\}^{\infty}_{i=0}$ is defined as
\begin{align}
t_{i+1}=t_i+T\label{eq:ri}
\end{align}
where $T$, a positive design parameter, is free.
\begin{remark}
{\rm We use the simple triggering mechanism \eqref{eq:ri} here for easier implementation and reducing reading burden.  A more complicated event-triggering mechanism in \cite{J2022Event} to update batch least-squares identifier for hyperbolic PDEs by evaluating the net increase of the system norms can also be used here, which can generate more effective triggering times but requires heavy computation to estimate the increase of the system norms by Lyapunov analysis.}
\end{remark}
\subsubsection{Batch least-squares identifier}\label{sec:ls}
According to   \eqref{eq:o2}, \eqref{eq:o3}, we get for $\tau>0$ and $\bar n =1,2,\cdots$ that
\begin{align}
&\frac{d}{d\tau}\left(\int_0^{1}\sin({x\pi \bar n})z(x,\tau)dx+\int_0^{1}\sin({x\pi \bar n})w(x,\tau)dx\right)\notag\\
&=-{q_2\pi \bar n}\int_0^{1}\cos({x\pi \bar n})w(x,\tau)dx+d_1\int_0^{1}\sin({x\pi \bar n})w(x,\tau)dx\notag\\
&+{q_1\pi \bar n}\int_0^{1}\cos({x\pi \bar n})z(x,\tau)dx\notag\\
&+d_2\int_0^{1}\sin({x\pi \bar n})z(x,\tau)dx,\label{eq:Ls1}\\
&\frac{d}{d\tau} y_n(\tau)=\sum_{i=1}^n l_iy_i(\tau)+bw(0,\tau).\label{eq:Ls2}
\end{align}
Define the instant $\mu_{i+1}$ as
\begin{align}
\mu_{i+1}=\min\{t_g:g\in\{0,\ldots,i\},t_g\ge t_{i+1}-\tilde N T\},\label{eq:mui}
\end{align}
for $i\in\mathbb Z^+$, where the positive integer $\tilde N\ge 1$ is a free design parameter.
Integrating \eqref{eq:Ls1}, \eqref{eq:Ls2} from $\mu_{i+1}$ to $t$,  yields
\begin{align}
&p_{\bar n}(t,\mu_{i+1})=d_1g_{\bar n,{1}}(t,\mu_{i+1})+d_2g_{\bar n,{2}}(t,\mu_{i+1}),\label{eq:f}\\
&p_{\rm b}(t,\mu_{i+1})=bq_{\rm b}(t,\mu_{i+1})\label{eq:p}
\end{align}
where $p_{\bar n}, g_{\bar n,{1}}, g_{\bar n,{2}}, p_{\rm b}, q_{\rm b}$ are
\begin{align}
&p_{\bar n}(t,\mu_{i+1})=\int_0^{1}\sin({x\pi \bar n})z(x,t)dx+\int_0^{1}\sin({x\pi \bar n})w(x,t)dx\notag\\
&-\int_0^{1}\sin({x\pi \bar n})z(x,\mu_{i+1})dx-\int_0^{1}\sin({x\pi \bar n})w(x,\mu_{i+1})dx\notag\\
&-{\pi \bar n}\int_{\mu_{i+1}}^t\int_0^{1}\cos({x\pi \bar n})(q_1z(x,\tau)-q_2w(x,\tau))dxd\tau\label{eq:fn}\\
&g_ {\bar n,1}(t,\mu_{i+1})=\int_{\mu_{i+1}}^t\int_0^{1}\sin({x\pi \bar n})w(x,\tau)dxd\tau\label{eq:gq1}\\
&g_ {\bar n,2}(t,\mu_{i+1})=\int_{\mu_{i+1}}^t\int_0^{1}\sin({x\pi \bar n})z(x,\tau)dxd\tau\label{eq:gq2}
\end{align}
for $\bar n=1,2,\cdots$, and
\begin{align}
p_{\rm b}(t,\mu_{i+1})=&y_n (t)- y_n (\mu_{i+1})-\int_{\mu_{i+1}}^{t}\sum_{i=1}^n l_iy_i(\tau) d\tau,\label{eq:f1}\\
q_ {\rm b}(t,\mu_{i+1})=&\int_{\mu_{i+1}}^{t}w(0,\tau)d\tau.\label{eq:ga}
\end{align}
Define the function $h_{i,\bar n}:\mathbb R^{3}\to \mathbb R_+$ by the formula
\begin{align}
&h_{i,\bar n}(\ell)=\int_{\mu_{i+1}}^{t_{i+1}}[(p_{\bar n}(t,\mu_{i+1})-\ell_1g_{\bar n,{1}}(t,\mu_{i+1})\notag\\&-\ell_2 g_{\bar n,{2}}(t,\mu_{i+1}))^2+(p_{\rm b}(t,\mu_{i+1})-\ell_3q_{\rm b}(t,\mu_{i+1}))^2]dt,\label{eq:hi}
\end{align}
for $i\in  \mathbb Z^+$, where $\ell=[\ell_1,\ell_2,\ell_3]^T$. According to \eqref{eq:f}, \eqref{eq:p}, the function $h_{i,\bar n}(\ell)$ in \eqref{eq:hi} has a global minimum $h_{i,\bar n}(\theta)=0$. We get from Fermat's theorem (vanishing gradient at extrema) that the following matrix equation holds for every $i\in  \mathbb Z_+$ and $\bar n\in \mathbb N$:
\begin{align}
Z_{\bar n}(\mu_{i+1},t_{i+1})=G_{\bar n}(\mu_{i+1},t_{i+1})\theta\label{eq:Fer}
\end{align}
where $\theta$ defined in \eqref{eq:theta} is a column vector of unknown parameters, and where
\begin{align}
Z_{\bar n}=[H_{{\bar n},1},H_{{\bar n},2},H_{3}]^T,~~G_{\bar n}=\left[
    \begin{array}{ccc}
      Q_{{\bar n},1} & Q_{{\bar n},2} & 0\\
     Q_{{\bar n},2} & Q_{{\bar n},3} & 0\\
     0 & 0 & Q_{4}\\
    \end{array}\label{eq:G}
  \right],
\end{align}
with $H_{{\bar n},1}, H_{{\bar n},2}, H_{3}, Q_{{\bar n},1} , Q_{{\bar n},2}, Q_{{\bar n},3},  Q_{4}$ defined as
\begin{align}
H_{\bar n,2}(\mu_{i+1},t_{i+1})&=\int_{\mu_{i+1}}^{t_{i+1}}g_{\bar n,{2}}(t,\mu_{i+1})p_{\bar n}(t,\mu_{i+1}) dt,\label{eq:H1m}\\
H_{\bar n,1}(\mu_{i+1},t_{i+1})&=\int_{\mu_{i+1}}^{t_{i+1}}g_{\bar n,{1}}(t,\mu_{i+1})p_{\bar n}(t,\mu_{i+1}) dt,\label{eq:H2m}\\
Q_{\bar n,3}(\mu_{i+1},t_{i+1})&=\int_{\mu_{i+1}}^{t_{i+1}}g_{\bar n,{2}}(t,\mu_{i+1})^2 dt,\label{eq:Q1m}\\
Q_{\bar n,2}(\mu_{i+1},t_{i+1})&=\int_{\mu_{i+1}}^{t_{i+1}}g_{\bar n,{2}}(t,\mu_{i+1})g_{\bar n,{1}}(t,\mu_{i+1}) dt,\label{eq:Q2m}\\
Q_{\bar n,1}(\mu_{i+1},t_{i+1})&=\int_{\mu_{i+1}}^{t_{i+1}}g_{\bar n,{1}}(t,\mu_{i+1})^2 dt,\label{eq:Q3m}\\
H_{3}(\mu_{i+1},t_{i+1})&=\int_{\mu_{i+1}}^{t_{i+1}}q_{\rm b}(t,\mu_{i+1})p_{\rm b}(t,\mu_{i+1}) dt,\label{eq:H3m}\\
Q_{4}(\mu_{i+1},t_{i+1})&=\int_{\mu_{i+1}}^{t_{i+1}}q_{\rm b}(t,\mu_{i+1})^2 dt.\label{eq:Q4}
\end{align}
Indeed, \eqref{eq:Fer}, \eqref{eq:G} are obtained by differentiating the functions $h_{i,{\bar n}}(\ell)$ defined by \eqref{eq:hi} with respect to $\ell_1$, $\ell_2$, $\ell_3$, respectively, and using the fact that the derivatives at the position of the global minimum $(\ell_1,\ell_2,\ell_3)=(d_1,d_2,b)$ are zero.

The parameter estimator (update law) is defined as
\begin{align}
&\hat\theta(t_{i+1})={\rm argmin}\bigg\{|\ell-\hat\theta(t_i)|^2: {\ell\in \Theta}, \notag\\
&Z_{\bar n}(\mu_{i+1},t_{i+1})=G_{\bar n}(\mu_{i+1},t_{i+1})\ell,~~ \bar n=1,2,\cdots\bigg\},\label{eq:adaptivelaw}
\end{align}
where
$${\Theta=\{\ell\in \mathbb R^3:\underline d_1\le \ell_1\le \overline d_1 ,\underline d_2\le \ell_2\le \overline d_2, 0<\underline b\le b\le \overline b\}}.$$
Please note that even though all states are measurable in the full-state feedback case, for the parameter estimation, we can not use the naive method--that is, taking the time and spatial derivatives of the measurements $z(x,t),w(x,t)$ and taking the time derivative of the measurement $Y(t)$, to calculate $d_1,d_2,b$ straightforwardly by \eqref{eq:o1}, \eqref{eq:o2}, \eqref{eq:o3}, because of the following two reasons: 1) taking the time derivative of the measured signals always leads to the undesired noise amplification in practice; 2) the possible zero values of $w(0,t)$, $z(x,t)$, $w(x,t)$ accompanied with the unknown parameters will cause singularity.
\subsection{Safe adaptive controller}
The adaptive controller \eqref{eq:Ud} is potentially unsafe because the mismatch between the parameter estimates and the true values leads to the safety obtained in Theorem \ref{th:th1} is not ensured anymore. Next, we design an output-positive adaptive controller by making use of a QP safety filter to override the adaptive controller $U_d$ \eqref{eq:Ud}.

Considering the plant parameters $d_1,d_2,b$ that are considered as unknown in this section, the design parameter conditions \eqref{eq:kappai}, \eqref{eq:ci} in the nominal control design are slightly modified as
\begin{align}
\kappa_i>&\max_{\underline b\le \varsigma\le \overline b}\check \kappa_i(\varsigma),~~i=1,\cdots,n-1\label{eq:kin}\\
c_i>&\max\left\{2, \max_{\vartheta\in\Theta}\check c_i(\vartheta)\right\},i=1,\cdots,m-1,~c_m>1\label{eq:cin}
\end{align}
using the known bounds of the unknown parameters in Assumption \ref{as:bound}, where $\check \kappa_i$, $\check c_i$ are defined in \eqref{eq:kappai}, \eqref{eq:underc}, respectively.

With \eqref{eq:kin}, \eqref{eq:cin}, and Assumption \ref{as:initial}, the required positive initialization in Lemmas \ref{cl:zi0}, \ref{cl:hi0} is achieved. Like the nominal control design, a sufficient condition to the safety guarantee, i.e, the non-negativity of the functions $h_i,\beta,z_i$, is
$h_m(t)\ge 0$ all the time, whose sufficient condition is
\begin{align}
\dot h_m(t;\theta)\ge -\bar c h_m(t;\theta)\label{eq:dymal}
\end{align}
under the positive initialization, where the positive constant $\bar c$ is a design parameter satisfying
$\bar c\ge c_m$.
Writing
$\theta=[d_1,d_2,b]^T$
after $``;"$ in $h_m(t;\theta)$ emphasizes the fact that now $h_m$ depends on the unknown parameters $d_1,d_2,b$.
Considering the uncertainties, recalling \eqref{eq:o7}, \eqref{eq:ncti}, \eqref{eq:taui} and the adaptive estimate \eqref{eq:adaptivelaw}, it follows the CBF constraint \eqref{eq:dymal} that a safe region of the adaptive control action is
\begin{align}
\mathcal C(t)=\left\{u\in\mathbb R:u\ge \max_{\vartheta\in D_i} U^{*}(\chi;\vartheta)\right\},~t\in[t_i, t_{i+1})\label{eq:Ustartb}
\end{align}
where the explicit function $U^{*}$ is
\begin{align}
U^{*}(\chi;\theta)=&(c_m-\bar c)h_m+\tau_m-{\sum_{i=0}^{m-1} \bar q_iz_t^{(i)}(1,t)-M^TY(t)}\notag\\&+\Gamma^{(m)}(t), \label{eq:Ustart}
\end{align}
which is identical to \eqref{eq:U1} when replacing $\bar c$ with $c_m$. The sets
$D_{i}$ are generated in running BaLSI \eqref{eq:adaptivelaw} and given by
\begin{align}\label{eq:Di}
D_{i}=\{\ell\in \Theta: Z_{\bar n}(\mu_{i},t_{i})=G_{\bar n}(\mu_{i},t_{i})\ell\}\cap D_{i-1}
\end{align}
for $i\in\mathbb N$ and $D_0=\Theta$.
It implies that
\begin{align}
\theta\subseteq D_i\subseteq\Theta, ~i\in\mathbb Z^+, \label{eq:setD}
\end{align}
recalling \eqref{eq:Fer}. Computing $\max_{\vartheta\in D_i} U^{*}(\chi;\vartheta)$ will not cost too much time because $U^{*}(\chi;\vartheta)$ is an explicit function of $\vartheta$, and the seeking range $D_i$ would possibly be shrunk as time goes on until it becomes a singleton that is equal to true value $\theta$, which will be shown in Lemma \ref{lem:estimate}.

Making use of a QP safety filter to guarantee the adaptive control input stays in the safe region \eqref{eq:Ustartb}, we build the output-positive adaptive controller $U_a(t)$ as
\begin{align}\label{eq:Uas}
U_a= \arg \min_{u\in \mathbb R}\{|u-U_d|^2\},~~s.t.~ u\in \mathcal C(t).
\end{align}
\begin{remark}\label{rem:control}
{\rm To avoid falling into the extreme and rare cases that hinder the finite time exact parameter identification, we add some tips, whose reasons will be seen clearly in the proof of  Lemma \ref{lem:estimate}, for implementation of the control law \eqref{eq:Uas}. 1) If $z[t]-pw[t]\equiv0$ ($w[t]$, $z[t]$ are not identically zero) on $t\in[\mu_{i+1},\frac{\mu_{i+1}+t_{i+1}}{2}]$, $i\in\mathbb Z_+$, $t_{i+1}\le t_f$, we should add a small excitation signal into the current control signal during $t\in[\frac{\mu_{i+1}+t_{i+1}}{2},t_{i+1})$, $i\in\mathbb Z_+$, $t_{i+1}\le t_f$, to avoid that $z[t]-pw[t]\equiv0$ on $t\in[\mu_{i+1},t_{i+1})$, $i\in\mathbb Z_+$, $t_{i+1}\le t_f$; 2) If $z[0]$ or $w[0]$ is identically zero, and $U_a(t)$ is identically zero on $t\in[\frac{1}{q_1},t_z)$ for a certain $t_z>\frac{1}{q_1}$,  then we should also add a small excitation signal into the current control signal for $t\in[t_z,t_{z+1})$, to make sure that $U_a(t)$ is not identically exact to zero for $t\in[\frac{1}{q_1},\infty)$. 3) If $w(0,t)\equiv0$ on $t\in[0,t_z)$ for a certain $t_z>0$, a small excitation signal should also be added into the current control signal during $t\in[t_z,t_{z+1})$ to avoid $w(0,t)\equiv 0$ for all time. The form of the excitation signal can be easily defined on the basis of $\varsigma_1(t),\varsigma_2(t)$ obtained in the proof Lemma \ref{lem:estimate}, and the values of excitation should be larger than zero in order to keep the control input staying in the safe region  \eqref{eq:Ustartb}.}
\end{remark}

Summary of the safe adaptive control design: we obtain an adaptive control law by replacing the unknown parameters in the nominal safe control law $U=\mathcal U(\chi(t);\theta)$ \eqref{eq:U1} with the parameter estimates $\hat\theta$ generated from a BaLSI \eqref{eq:adaptivelaw} that ensures the exact parameter identification in finite time which will be shown in Lemma \ref{lem:estimate}. Even though the resulting adaptive control law $U_d$ \eqref{eq:Ud} can exponentially regulate the plant with uncertainties, it is potentially unsafe because it deviates from the nominal safe control law  in the adaptive learning process due to the parameter estimation errors. Therefore, we build the aCBF constraint to override the desired but potentially unsafe adaptive control law $U_d$ via the QP program \eqref{eq:Uas}. Once the exact parameter identification is achieved, the set $D_i$ generated in BaLSI is identically a singleton that is nothing else but $\theta$, which results in the fact that the safe adaptive control law $U_a$, the adaptive control law $U_d$, and the nominal control law  $U$ become identical, guaranteeing both the exponential regulation and safety as shown in Theorem \ref{th:th1}. The diagram of the safe-adaptive closed-loop system is shown in Fig. \ref{fig:closeloop}.
\begin{figure}
\centering
\includegraphics[width=8cm]{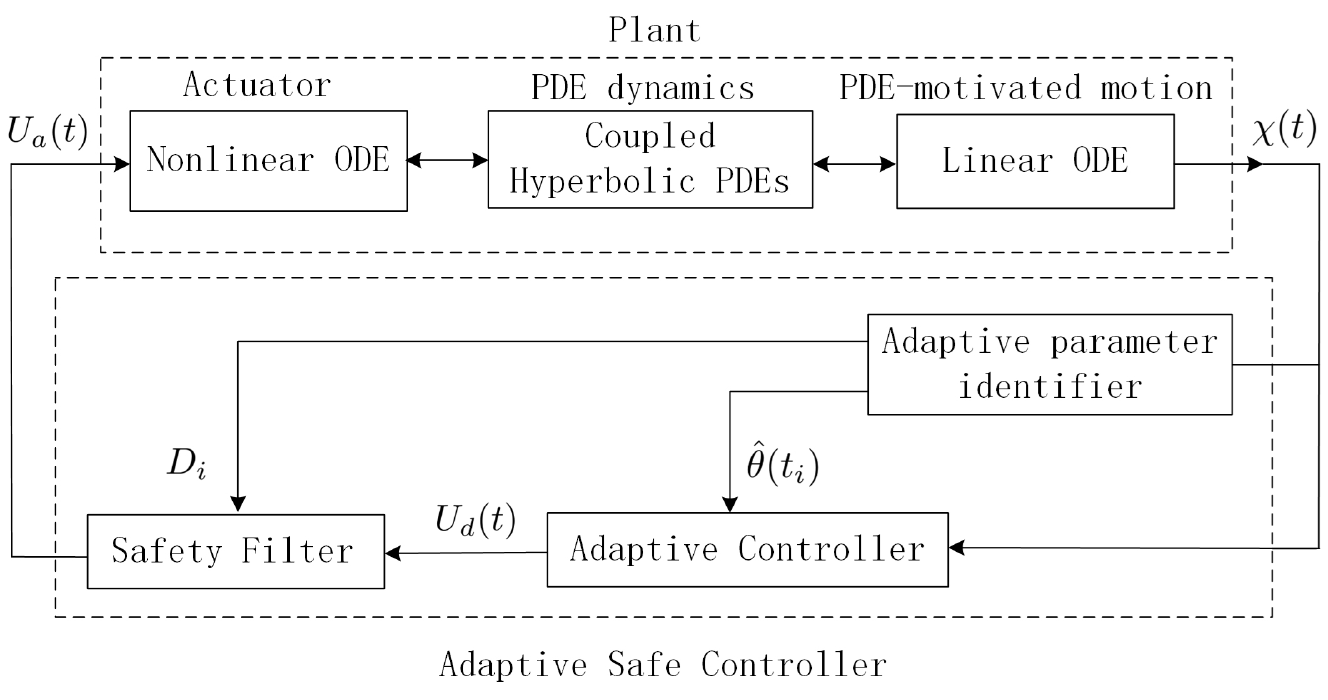}
\caption{The diagram of the proposed safe-adaptive control.}
\label{fig:closeloop}
\end{figure}
\begin{remark}
\emph{Our design encompasses all the essential features of a CBF safety design. The chain structures for the two ODEs,
 i.e., \eqref{eq:targ5}, \eqref{eq:targetym-1}, \eqref{eq:targetym} in the target system, are
essentially the high relative-degree CBFs, which were introduced in 2006 \cite{KrsticBement2006} for non-overshooting control and were independently discovered in 2016 \cite{2016Nguyen}. Moreover, as is typical with traditional CBFs for ODEs, the PDE state $\beta(x,t)$ in the target system is also required to be ensured non-negative for the purpose of safety. Additionally, the CBF-based control is often accompanied by a quadratic program (QP) safety filter, which is \eqref{eq:Uas} above.}
\end{remark}
\subsection{Result with safe adaptive control}
Defining the difference between $U(t)$ and $U_a(t)$ as $\eta(t)$, i.e,
\begin{align}
\eta(t)&=U(t)-U_a(t)\notag\\&=\mathcal U(\chi(t);\theta)-\mathcal U(\chi(t);\hat\theta(t_i)),~~t\in[t_i,t_{i+1})\label{eq:etaU}
\end{align}
inserting $U_a$ defined by \eqref{eq:Ud} into \eqref{eq:o7}, recalling the nominal control design in Sec. \ref{sec:ODEcontroldesign}, the target system becomes \eqref{eq:target1}--\eqref{eq:targetym-1}, together with
\begin{align}
\dot h_m(t)=-c_m h_m(t)-\eta(t).\label{eq:yma}
\end{align}
\begin{remark}\label{eq:reta}
\emph{The signal $\eta(t)$ results from the parameter estimation error. It only includes the signals from the $w, z$-PDEs (or $\beta, \alpha$ in terms of target states) and the distal $Y$-ODE  (or $Z$ in terms of target states), without the states in the $X$-ODE (or $h_i$, $i=1,\cdots, m$ in terms of target states), because the unknown parameter only exists in the linear subsystem of $w,z$-PDE and $Y$-ODE.}
\end{remark}
{\begin{prope}\label{Pro:1}
For every $(z[0],w[0],X(0),Y(0))\in  C^{m-1}([0,1])^2\times\mathbb R^m\times\mathbb R^n$,  there exists  a unique solution $(z,w,X,Y)\in  C^{m-1}([0,\infty)\times[0,1])^2\times\mathbb C^0 ([0,\infty);\mathbb{R}^{m})\times\mathbb C^0([0,\infty); \mathbb{R}^{n})$ to  the system \eqref{eq:o1}--\eqref{eq:o7} with the control input \eqref{eq:Uas}.
\end{prope}}
\begin{proof}
From $(z[0],w[0],X(0),Y(0))\in C^{m-1}([0,1])^2\times\mathbb R^m\times\mathbb R^n$, we obtain $(\alpha[0],\beta[0],H(0),Y(0))\in C^{m-1}([0,1])^2\times\mathbb R^m\times\mathbb R^n$ recalling the transformations \eqref{eq:contran1a}, \eqref{eq:contran1b}, \eqref{eq:ncti}--\eqref{eq:nctG}, where $H(0)=[h_1(0),\ldots,h_m(0)]^T$.

Timing $e^{c_m t}$ into both sides of \eqref{eq:yma}, and integrating from $0$ to $t$, we directly have the solution
$h_m(t)=h_m(0)e^{-c_m t}+\int_0^{t}e^{c_m(\tau-t)}\eta(\tau)d\tau$
on $t\in[0,\infty)$ is continuous in the weak sense. Similarly, we have $h_i(t)\in C^0([0,\infty);\mathbb{R})$, $i=1,2,\cdots,m-1$ on $t\in[0,\infty)$ considering \eqref{eq:target5}, \eqref{eq:targetym-1}. It implies that $h_1(t)\in C^{m-1}([0,\infty);\mathbb{R})$. Define $ h_1^{(m-1)}(t)=d(t)$ that implies $d(t)\in C^0([0,\infty);\mathbb{R})$, and $\beta_t^{(m-1)}[t]=u[t]$ that implies $u[0]\in C^{0}([0,1])$ considering $\beta[0]\in C^{m-1}([0,1])$.  Taking the $m-1$ time derivatives of \eqref{eq:target5}, \eqref{eq:target3} and rewriting them in term of $d,u$, recalling Theorem 2.6 in \cite{2019Karafyllis}, we have $u\in C^{0}([0,\infty)\times[0,1])$ that means that $\beta\in C^{m-1}([0,\infty)\times[0,1])$.  Similarly, recalling \eqref{eq:target2} and $\alpha[0]\in C^{m-1}([0,1])$, we also have $\alpha\in C^{m-1}([0,\infty)\times[0,1])$. Timing $e^{-A_{\rm Z}t}$ into both sides of \eqref{eq:targ5}, and integrating from $0$ to $t$, we have
$Z(t)=Z(0)e^{A_{\rm Z}t}+\int_0^{t}e^{A_{\rm Z}(t-\tau)}B\beta(0,\tau)d\tau$,
where $\beta(0,t)$ is continuous. It is straightforward to  have $Z\in C^0([0,\infty),\mathbb{R}^{n})$ in the weak sense.
Recalling the transformations \eqref{eq:ncti}--\eqref{eq:nctG}, \eqref{eq:Icontran1a}, \eqref{eq:Icontran1b}, \eqref{eq:ZTY}, we then obtain $(z,w,X,Y)\in C^{m-1}([0,\infty)\times[0,1])^2\times\mathbb C^0 ([0,\infty);\mathbb{R}^{m})\times\mathbb C^0([0,\infty);\mathbb{R}^{n})$.
\end{proof}
\begin{lema}\label{lem:estimate}
The finite-time exact identification of the unknown parameters is achieved, that is, there exists a finite updating time $t_f$, $f\in\mathbb Z^+$, such that
$\hat\theta (t)\equiv\theta,~\forall t\ge t_f$. Also, the set $D_i$ defined by \eqref{eq:Di} is shrunk to a singleton at $t_f$ and is kept at the singleton, which is nothing else but the unknown parameter's true value $\theta$, for $i\ge f$, $i\in\mathbb Z^+$.
\end{lema}
\begin{proof}
The proof is given in Appendix\ref{AP:convergence}.
\end{proof}
The convergence time $t_f$ can be influenced by the free positive design parameter $T$ that is related to the amount of the measurement data used in parameter estimation. The larger $T$ improves the robustness to the measurement error but prolongs the time till exact parameter identification. On the contrary, the smaller $T$ contributes to the fast adaption, which is desired in safe adaptive control; however, the robustness to the measurement error may be reduced.
\begin{thme}\label{thm:safeadaptive}
For initial data $(w[0],z[0])^T\in C^{m-1}([0,1])$, ${\hat\theta(0)\in \Theta}$, $X(0)\in \mathbb R^m$, $Y(0)\in \mathbb R^m$ satisfying Assumptions \ref{as:f}--\ref{as:initialx1}, for design parameters $c_i$, $i=1,\ldots,m$ satisfying \eqref{eq:cin} and  $\kappa_i$, $i=1,\ldots,n-1$ satisfying \eqref{eq:kin}, the closed-loop system including the plant \eqref{eq:o1}--\eqref{eq:o7} with the safe (output-positive) adaptive controller \eqref{eq:Uas} has the following properties:

1) Safety (output positivity) is ensured in the sense that
$y_1(t)\ge0,~\forall t\ge 0$. Moreover, it runs in the original safe set like the nominal safe control after the finite time.

2) Exponential regulation of the plant states is achieved in the sense that $\|w(\cdot,t)\|^2+\|z(\cdot,t)\|^2+|X(t)|^2+|Y(t)|^2$ is exponentially convergent to zero.

3) The safe (output-positive) adaptive control input is exponentially convergent to zero, i.e., $\lim_{t\to \infty} U_a(t)=0$.
\end{thme}
\begin{proof}
 1) According to \eqref{eq:Ustartb}--\eqref{eq:setD}, the  controller \eqref{eq:Uas} with the adaptive CBF constraint guarantees $\dot h_m(t)=-\bar ch_m(t)+\bar\eta(t)$ where $\bar\eta(t)=U_a(t)-U^{*}(\chi;\theta)\ge 0$ because of $U_a\ge \max_{\vartheta\in D_i} U^{*}(\chi;\vartheta)\ge U^*(\chi;\theta)$ due to \eqref{eq:setD}. It implies that   $h_m(t)\ge\bar\eta(t)\ge 0$ all the time recalling $h_m(0)>0$ ensured by the choices of design parameters \eqref{eq:cin}. Especially, for $t\in[t_f,\infty)$, it follows from Lemma \ref{lem:estimate} that $U_d(t)=U(t)$ recalling the nominal controller $U(t)$ \eqref{eq:U1} and the adaptive controller $U_d$ \eqref{eq:Ud}, as well as the control input's safe region boundary $\max_{\vartheta\in D_i} U^{*}(\chi;\vartheta)=U^{*}(\chi;\theta)$ in \eqref{eq:Ustart}. Then we have $\max_{\vartheta\in D_i} U^{*}(\chi;\vartheta) - U_d(t)=U^{*}(\chi;\theta) - U(t)=(c_m-\bar c)h_m(t)\le 0$ for $t\ge t_f$ recalling $\bar c\ge c_m$ and $h_m(t)\ge 0$ proven above. It implies that $U_d\in \mathcal C$ because of \eqref{eq:Ustartb}, and thus $U_a=U_d=U$ for $t\ge t_f$ according to \eqref{eq:Uas}, which means that
 \begin{align}
 \eta(t)\equiv 0,~\forall t\ge t_f\label{eq:eta0}
\end{align}
 in \eqref{eq:yma} because of \eqref{eq:eta}. Therefore, $\dot h_m(t)=-c_m h_m(t)$ holds on $t\in[t_f,\infty)$, i.e., the state runs in the original safe set like the nominal safe control after the finite time $t_f$. The property 1 in this theorem is thus obtained.

2)
Considering the Lyapunov function \eqref{eq:V},  choose the analysis parameters $a_0$ and $r$ as \eqref{eq:ra} and \eqref{eq:r}, where the upper bound $\overline b$ is used to replace the unknown $b$ in \eqref{eq:ra}, i.e.,
${a_0}>{\frac{q_1p^2}{q_2}}+\frac{{4|P\overline B|^2}}{{q_2}\lambda_{\rm min}(Q)}$ with $\overline B=[0,0,\ldots,\overline b]^T$.

Taking the time derivative of $V(t)$ along  the target system \eqref{eq:target1}--\eqref{eq:targetym-1}, \eqref{eq:yma},
 recalling $c_i$ defined by \eqref{eq:cin}, through a similar process with \eqref{eq:dLy}--\eqref{eq:sigma0}, we thus arrive at
\begin{align}
\dot V(t)\le -\sigma_0 V(t)-rh_m(t)\eta(t),\label{eq:dLyaa}
\end{align}
for $t\ge 0$, where $\sigma_0$ is defined in \eqref{eq:sigma0}.

According to \eqref{eq:eta0}, it follows \eqref{eq:dLyaa} that
$\dot V(t)\le -\sigma_0 V(t)$, $\forall t\ge t_f$,
that is,
\begin{align}
V(t)\le V(t_f)e^{-\sigma_0 (t-t_f)},~~\forall t\ge t_f.\label{eq:Vgetf}
\end{align}
Considering Remark \ref{eq:reta}, the target system \eqref{eq:target1}--\eqref{eq:targetym-1}, \eqref{eq:yma} is a fully linear ODE-PDE-ODE coupled system.
It should be noted that the
transient of such a linear system in the finite time $t\in[0,t_f)$ can be bounded by an arbitrarily fast decay rate
considering a trade off between the decay rate and the overshoot coefficient, i.e., the
higher the decay rate, the higher the overshoot coefficient.
Therefore, recalling \eqref{eq:Xibound}, we arrive at
$$\Xi(t)=\Upsilon\Xi(0)e^{-\sigma_0 t},~~\forall t\ge 0$$
for some positive $\Upsilon$.
Following the proof below \eqref{eq:targetstability} in the proof of the property 2 in Theorem \ref{th:th1}, we obtain the property 2 in this theorem.

 3) The proof is similar to that of the property 3 in Theorem \ref{th:th1}.
\end{proof}
\section{Simulation}
\subsection{Model}\label{sec:simmod}
The considered simulation model is \eqref{eq:o1}--\eqref{eq:o7} with the parameters
$A=[0,1;l_1,l_2]=[0,1;1,-0.5], B=[0,b]^T=[0,1]^T, q_1=q_2=1, d_1=0.8,
d_2=1, m=2, n=2, p=1$, $\bar q_0=1$, $\bar q_1=1$, $M^T=[0.1,0.3]$,
and the functions $f_1$ and $f_2$ in \eqref{eq:o6}, \eqref{eq:o7} are
$f_1(\underline x_1)=x_1^2,~f_2(\underline x_2)=x_1x_2$.
The known bounds of the unknown parameters $d_1,d_2,b$ are set as
$\overline d_1=\overline d_2=1.2,~\underline d_1=\underline d_2=0.2,~\overline b=1.5,~\underline b=0.5$.
The initial values are defined as
$w(x,0)=\cos(2\pi x),~z(x,0)=2\sin(3\pi x),~x_1(0)=1,~x_2(0)=-1,~y_1(0)=5,~y_2(0)=0,~\hat d_1(0)=\hat d_2(0)=0.2,~\hat b(0)=0.5$. {The initial values in the simulation satisfy Assumptions \ref{as:initial}, \ref{as:initialx1}. The simulation model is open-loop unstable, where the unstable sources exist in every subsystem.}
\subsection{Controller}
Following the control design in Sec. \ref{sec:TandB}, we obtain the nominal controller as
\begin{align}
&U(t)
=-c_2x_2(t)-c_2c_1x_1(t)-c_2x_1^2-x_1x_2-c_1(x_{2}+x_1^2)\notag\\&+\int_0^1 \bigg(c_2c_1{\Psi}(1,y)+(c_1+c_2)R_1(y)+R_2(y)\bigg)z(y,t)dy \notag\\
&+  \int_0^1 \bigg(c_2c_1{\Phi}(1,y)+(c_1+c_2)P_1(y)+P_2(y)\bigg)w(y,t)dy\notag\\&-\bigg((c_1+c_2)q_1R_{0}(1)+{\bar q_0}+q_1R_{1}(1)\bigg)z(1,t)\notag\\&+\bigg((c_1+c_2)q_1R_{0}(0)+q_1R_{1}(0)\bigg)z(0,t)\notag\\
&+\bigg((c_1+c_2)q_2P_{0}(1)+q_2P_{1}(1)\bigg)w(1,t)\notag\\&-\bigg((c_1+c_2)\left(q_2P_{0}(0)-\lambda (1)B\right)+q_2P_{1}(0)-\lambda (1)AB\bigg)w(0,t)\notag\\&+ \bigg[\lambda (1)\bigg((c_1+c_2)A+c_2c_1+ A^{2}\bigg){-M^T}\bigg]{Y}(t)\notag\\&-(q_1R_{0}(1)+{\bar q_1})z_{t}(1,t)+q_1R_{0}(0)z_{t}(0,t)+q_2P_{0}(1)w_{t}(1,t)
\notag\\&-\left(q_2P_{0}(0)-\lambda (1)B\right)w_{t}(0,t)\label{eq:Usim}
\end{align}
where $\lambda (x)$ is given by \eqref{eq:kerf} with $K$ defined in \eqref{eq:K} as  $$K^T=-\frac{1}{b}[l_1+\kappa_1\kappa_2, l_2+\kappa_1+\kappa_2],$$ and where
\begin{align*}
R_0(y)&={\Psi}(1,y),~P_0(y)={\Phi}(1,y),\\
R_1(y)
&=q_1{\Psi}_y(1,y)+d_2{\Phi}(1,y),\\
P_1(y)
&=-q_2{\Phi}_y(1,y)+d_1{\Psi}(1,y),\\
R_2(y)
&=q_1^2{\Psi}_{yy}(1,y)+(q_1-q_2)d_2{\Phi}_y(1,y)+d_1d_2{\Psi}(1,y),\\
P_2(y)
&=q_2^2{\Phi}_{yy}(1,y)+(q_1-q_2)d_1{\Psi}_y(1,y)+d_1d_2{\Phi}(1,y).
\end{align*}
The explicit solutions of ${\Psi}(1,y),{\Phi}(1,y)$  are given in \eqref{eq:exk1}, \eqref{eq:exk2} and \eqref{eq:traker3}, \eqref{eq:F}, \eqref{eq:H}.

Based on the nominal safe control \eqref{eq:Usim}, following the control design in Sec. \ref{sec:adaptive}, the safe adaptive controller \eqref{eq:Uas} is obtained. In details,
the parameter estimator $\hat\theta(t)$ \eqref{eq:adaptivelaw} is obtained by using \eqref{eq:mui}, \eqref{eq:fn}--\eqref{eq:ga} and \eqref{eq:G}--\eqref{eq:Q4}. The switching time $t_f$ is $t_f=t_1=T$ considering the nonzero initial values $w(x,0),z(x,0)$ given in Sec. \ref{sec:simmod}.
Inserting the estimate $\hat\theta(t)$ into \eqref{eq:Usim} to replace the unknown $\theta=[d_1,d_2,b]^T$, we thus obtain the (potentially unsafe) adaptive controller $U_d(\chi;\hat \theta)$. Recalling \eqref{eq:Ustart}, we obtain $U^*$ that is the result of replacing $c_2$ in \eqref{eq:Usim} with $\bar c$. Considering $t_f=t_1$ and $C_0=\Theta$, the signal $\max_{\vartheta\in C_0} U^{*}(\chi;\vartheta)$  is approximately treated as seeking the maximum of $U^*(t;\vartheta)$ with respect to $\vartheta$ on a gridded $\Theta$, where the range between the known upper and lower bounds given in Sec. \ref{sec:simmod} is divided by the interval of 0.2.
We choose the design parameters $c_1,c_2,\kappa_1,\kappa_2,\bar c, \bar n, \tilde N, T$ in the nominal safe controller and safe adaptive controller (where the nominal controller only requires $c_1,c_2,\kappa_1,\kappa_2$) as follows. The design parameters $c_1=38,c_2=20,\kappa_1=30,\kappa_2=10$ are chosen to satisfy the condition \eqref{eq:kin}, \eqref{eq:cin} with \eqref{eq:kappai}, \eqref{eq:underc}, and the arbitrary positive design parameters  $\bar c=1,\bar n=1, \tilde N=10, T=1.5$ are chosen considering the following trade-off in implementation. Increasing the design parameter
$T,\tilde N$ will let more measured data take part in estimation, which is helpful to improve the estimation accuracy but
will prolong the duration of adaptive learning. The increase of $\bar n$ contributes to finding the true values on time at the cost of more computing resources. The adjustment of $\bar c$ will affect the adaptive control input's safe region adopted in the QP safety filter.

Tips on simulation of the parameter identifier \eqref{eq:adaptivelaw}: 1) Approximating the integration with respect to the space variable in the identifier \eqref{eq:adaptivelaw} as a summation operator affects the parameter identification precision of the unknown parameters $d_1,d_2$ in the PDE domain. The smaller space step adopted in the simulation will make the error smaller. Besides, a larger design parameter $T$ can reduce the effect of this approximation error
to some extent but will slow the adaption; 2) If the difference between the estimates from the identifier at two adjacent updating times is smaller than $5\%$ of the true value, we consider that this difference is caused by the approximation error in the simulation, and thus keep the estimated value as same as the one at the former updating time.
\subsection{Simulation Result}
We conduct the simulation by the finite-difference method with a time step of 0.001 and a space step of 0.002 for the adaptive case, where the relatively small space step is selected to reduce the approximation error of integration in the identifier mentioned above, for the nominal case a larger space step as 0.05 can be used to save the computation time.
\begin{figure}
\centering
\includegraphics[width=7cm]{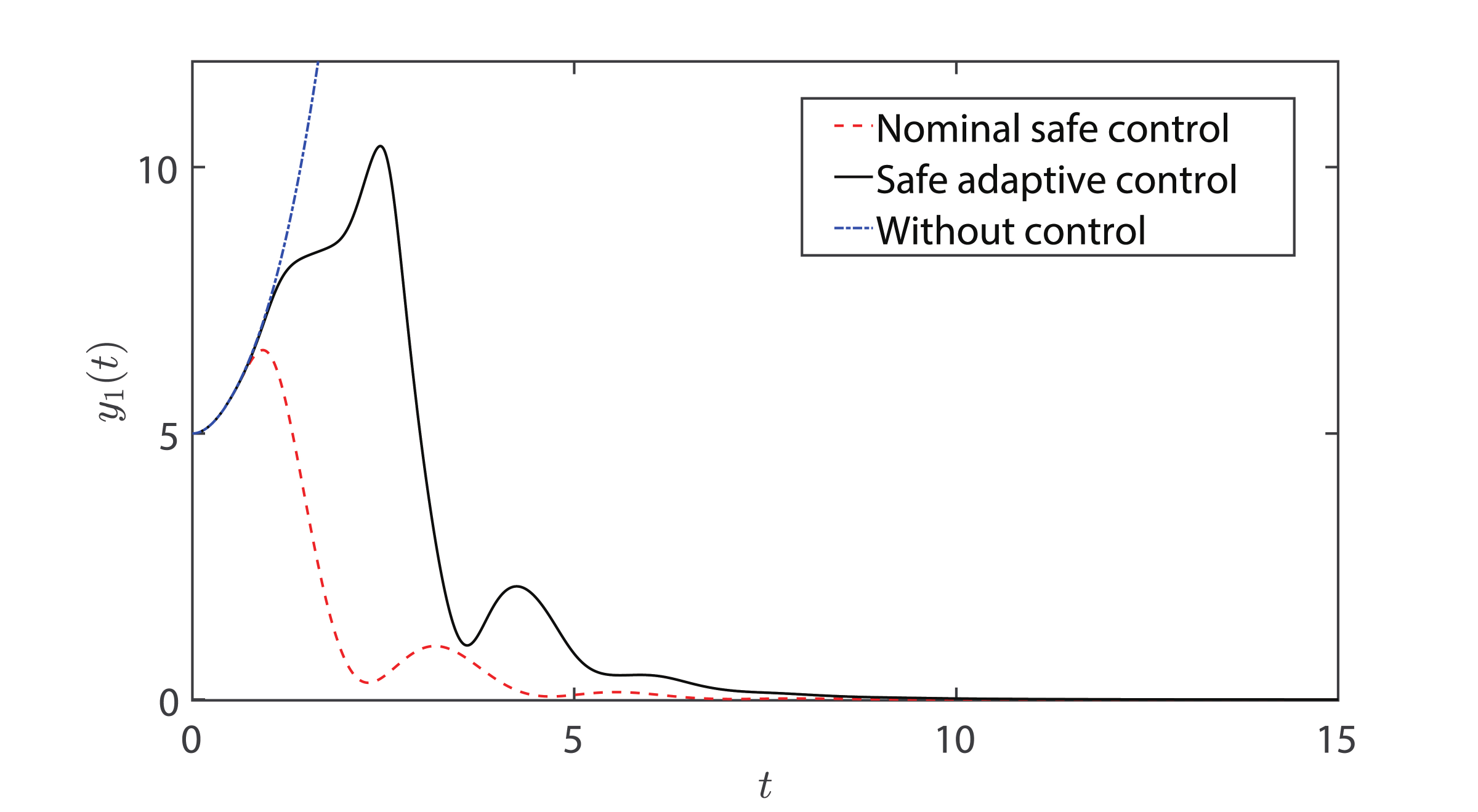}
\caption{Responses of the state to be safely regulated.}
\label{fig:y1}
\end{figure}
\begin{figure}
\begin{minipage}{0.49\linewidth}
  \centerline{\includegraphics[width=4.8cm]{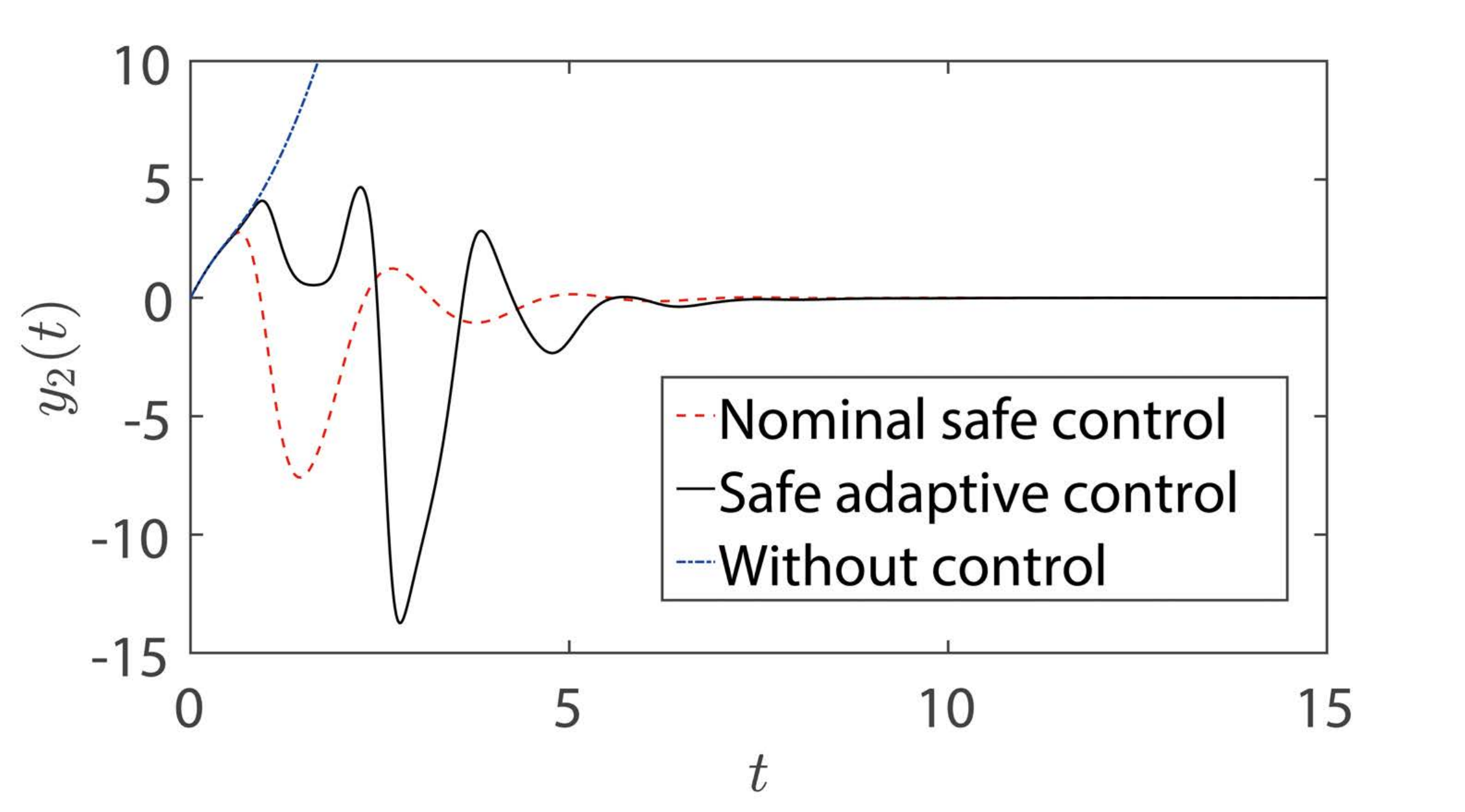}}
  \centerline{(a) $y_2(t)$}
\end{minipage}
\hfill
\begin{minipage}{.49\linewidth}
  \centerline{\includegraphics[width=4.8cm]{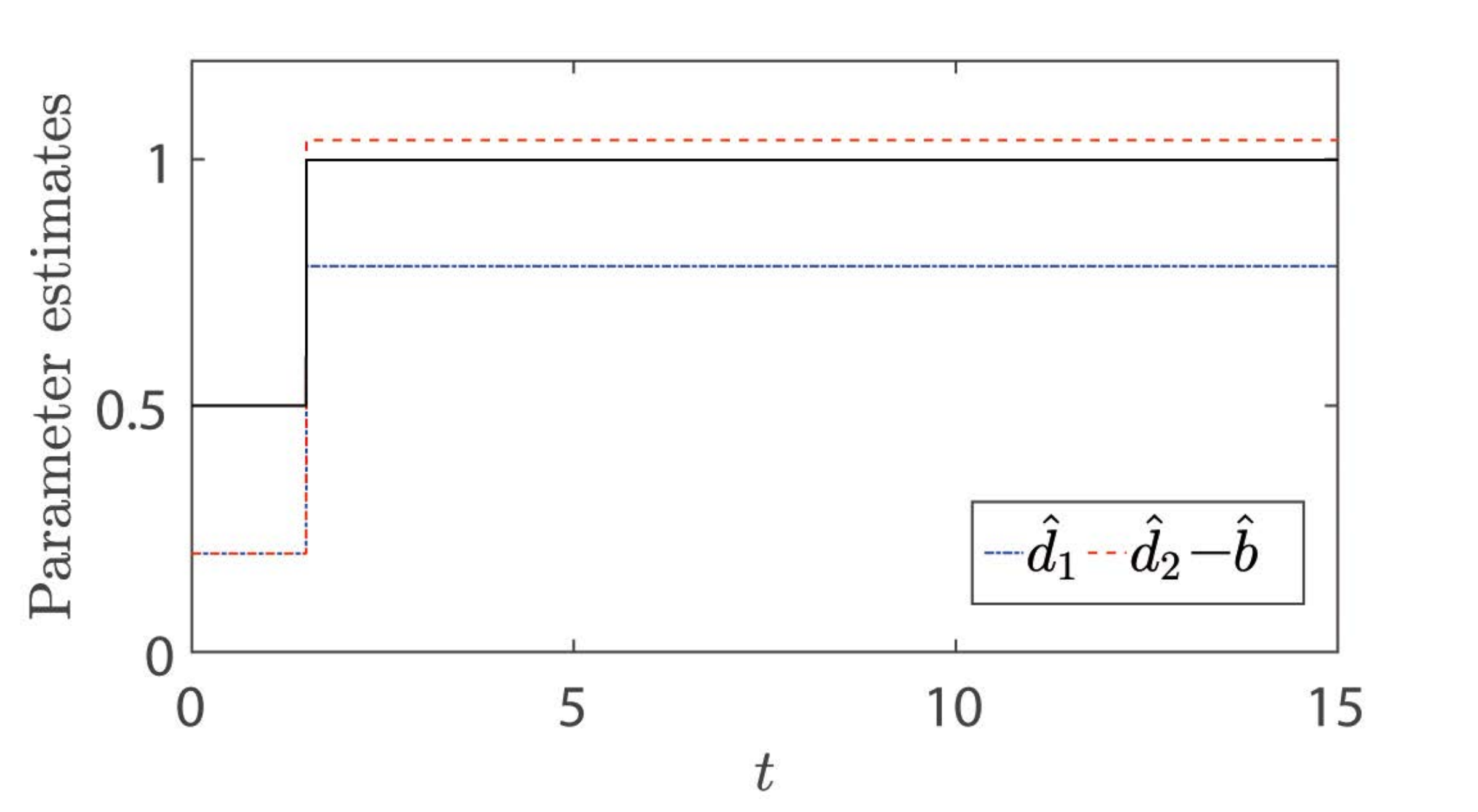}}
  \centerline{(b) estimates}
\end{minipage}
\caption{Responses of $y_2(t)$ and the parameter estimates.}
\label{fig:y2}
\end{figure}
\begin{figure}
\begin{minipage}{0.49\linewidth}
  \centerline{\includegraphics[width=4.7cm]{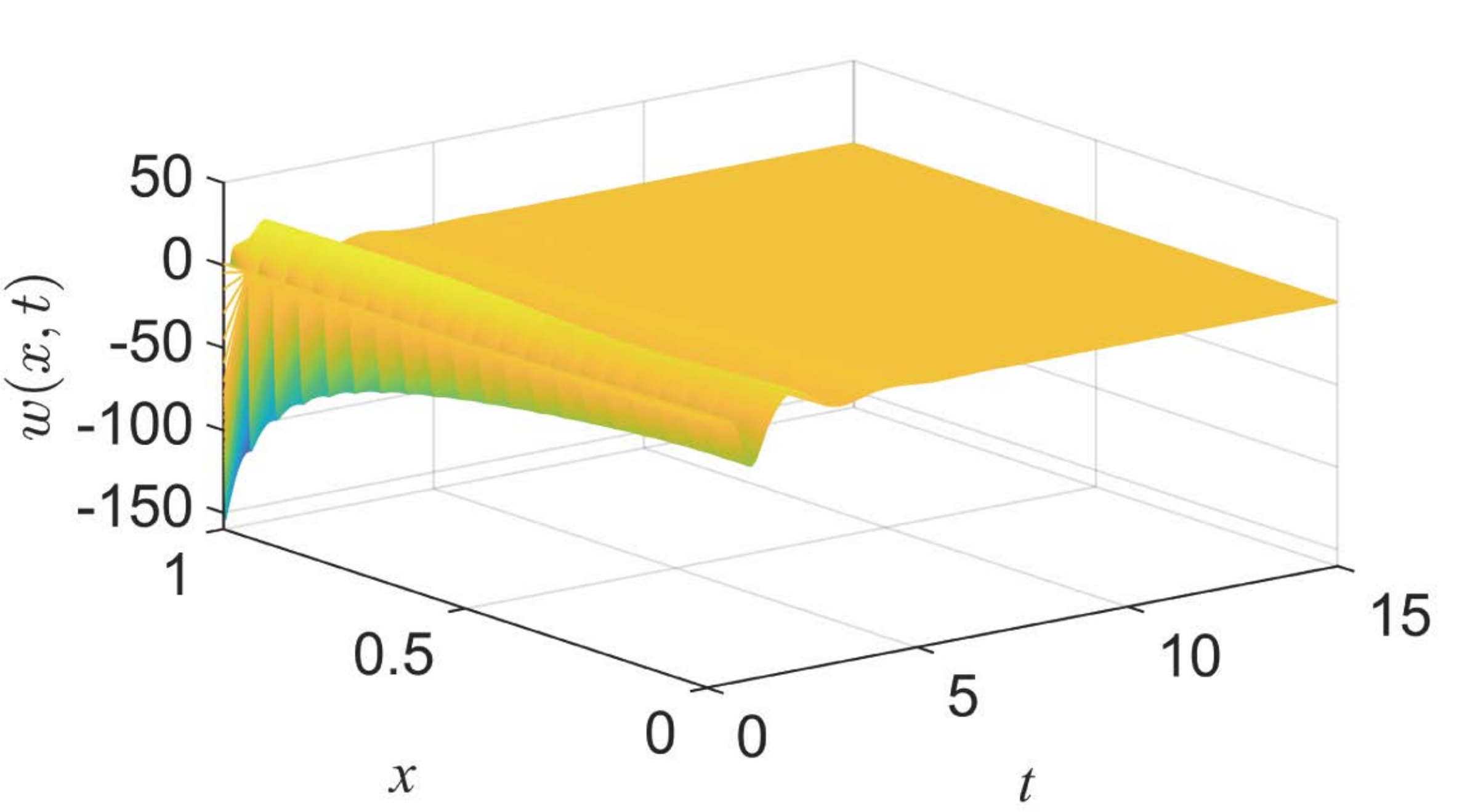}}
  \centerline{(a) Nominal safe control}
\end{minipage}
\hfill
\begin{minipage}{.49\linewidth}
  \centerline{\includegraphics[width=4.7cm]{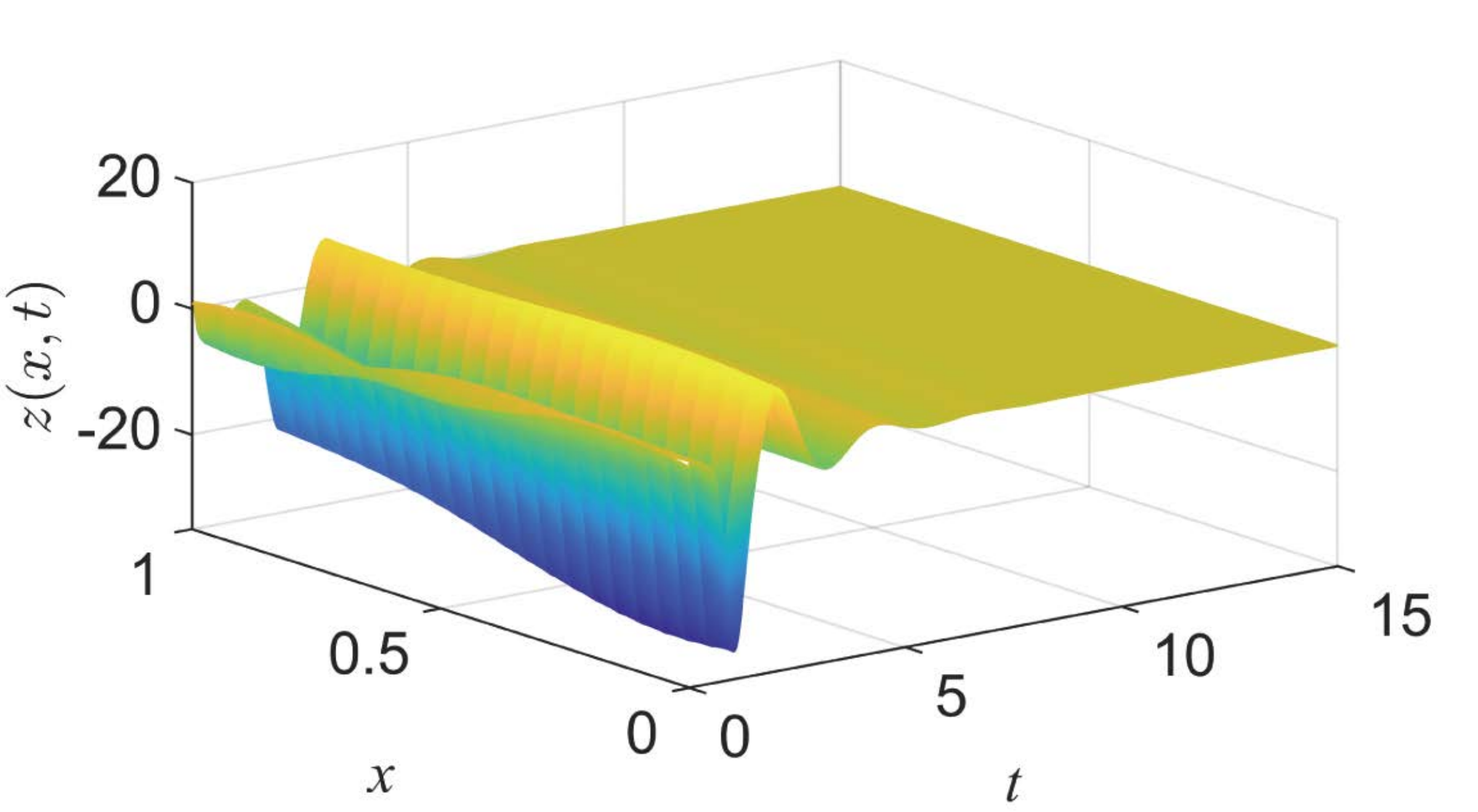}}
  \centerline{(b)  Nominal safe control}
\end{minipage}
\vfill
\begin{minipage}{0.49\linewidth}
  \centerline{\includegraphics[width=4.7cm]{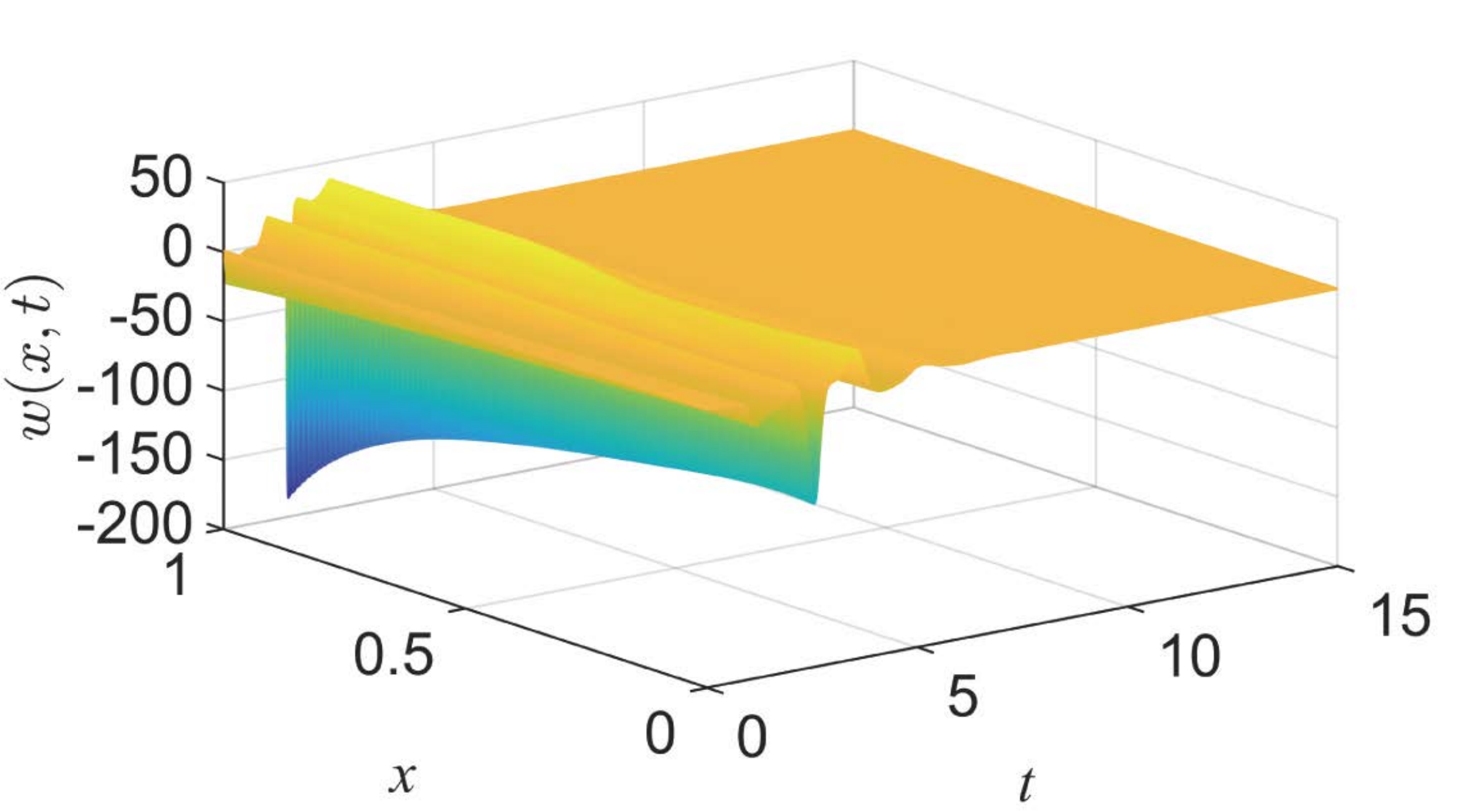}}
  \centerline{(c)  Safe adaptive control}
\end{minipage}
\hfill
\begin{minipage}{0.49\linewidth}
  \centerline{\includegraphics[width=4.7cm]{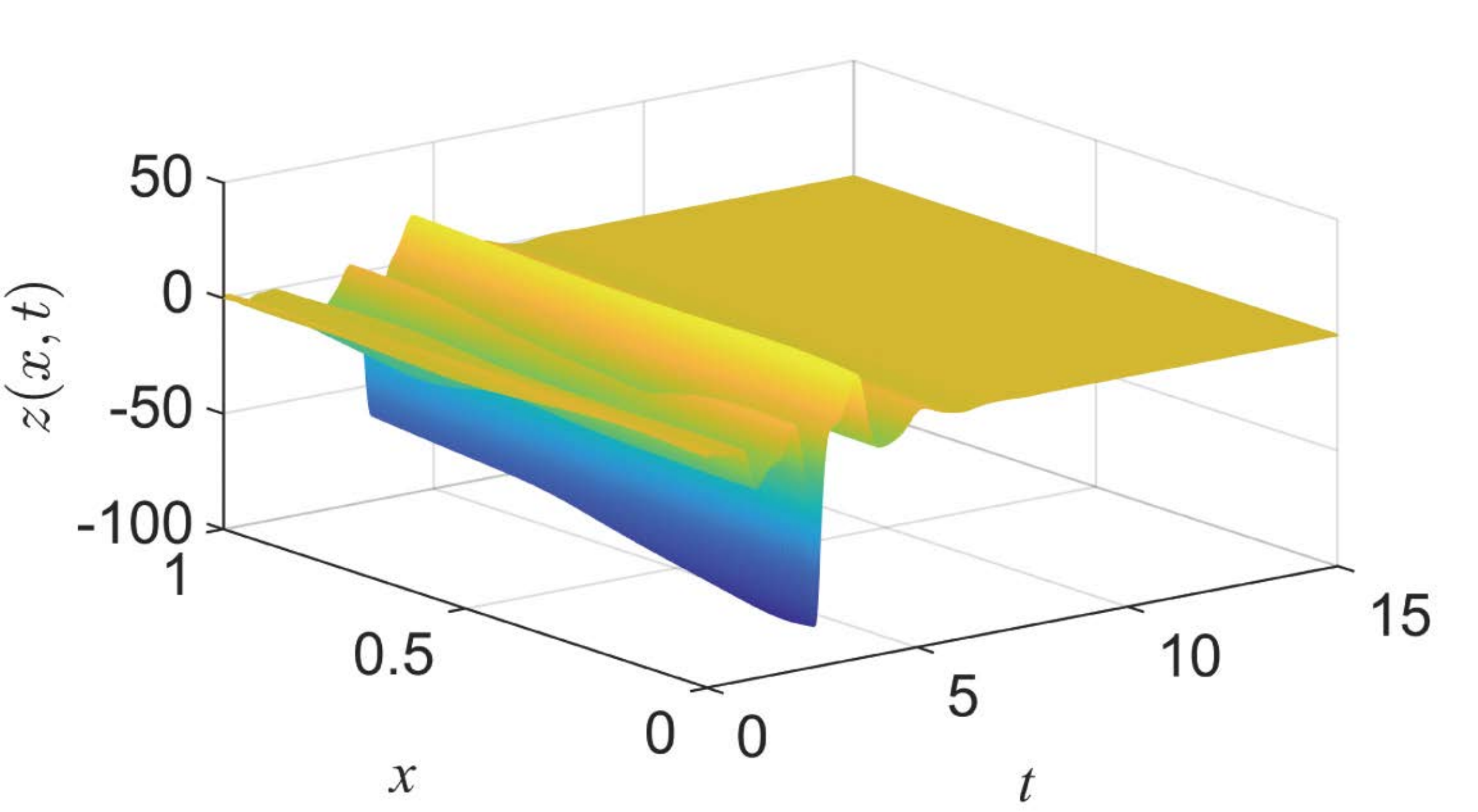}}
  \centerline{(d) Safe adaptive control}
\end{minipage}
\caption{Responses of $w(x,t),z(x,t)$ under the nominal safe
and safe adaptive controllers.}
\label{fig:wz}
\end{figure}
\begin{figure}
\begin{minipage}{0.49\linewidth}
  \centerline{\includegraphics[width=4.8cm]{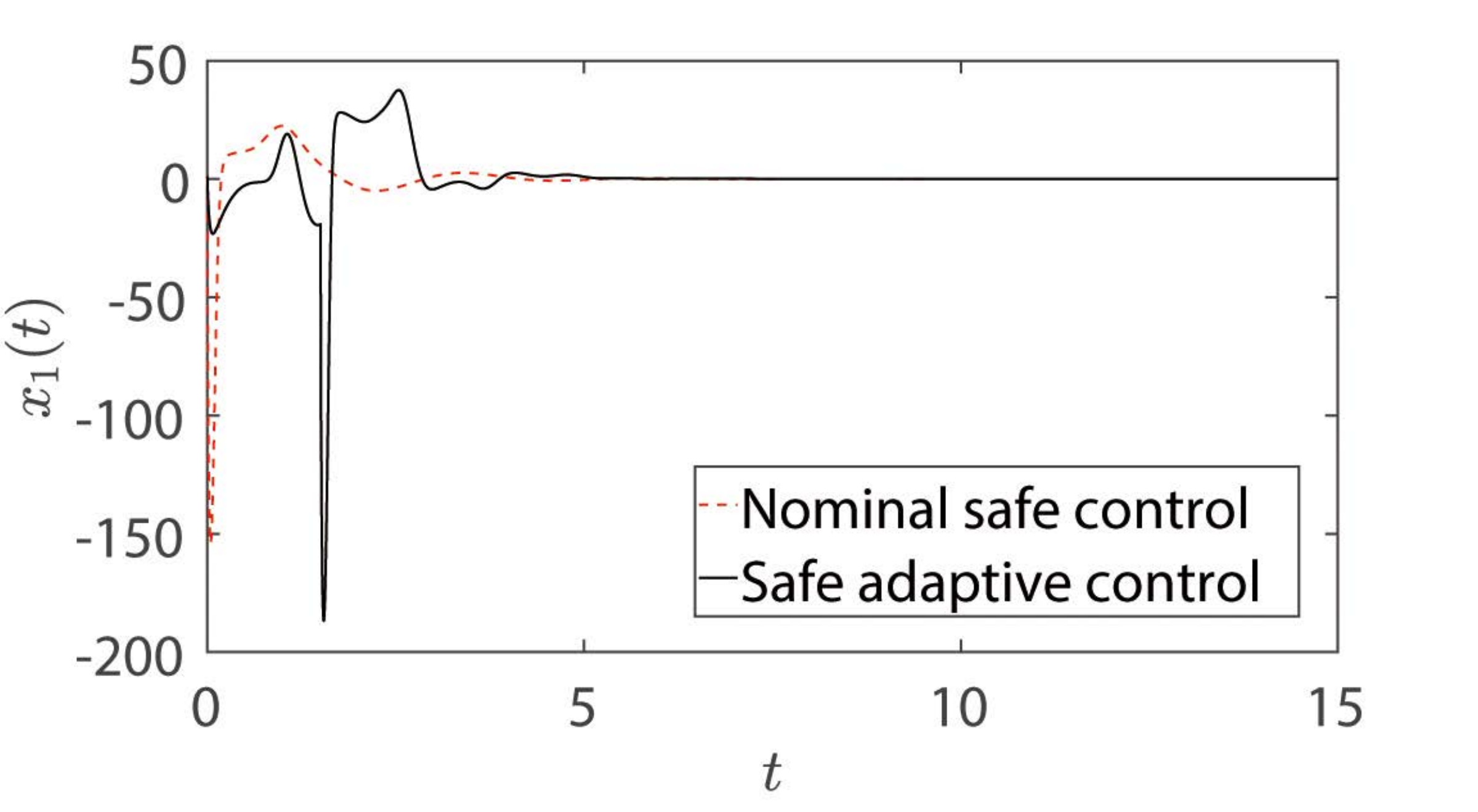}}
  \centerline{(a) $x_1(t)$}
\end{minipage}
\hfill
\begin{minipage}{.49\linewidth}
  \centerline{\includegraphics[width=4.8cm]{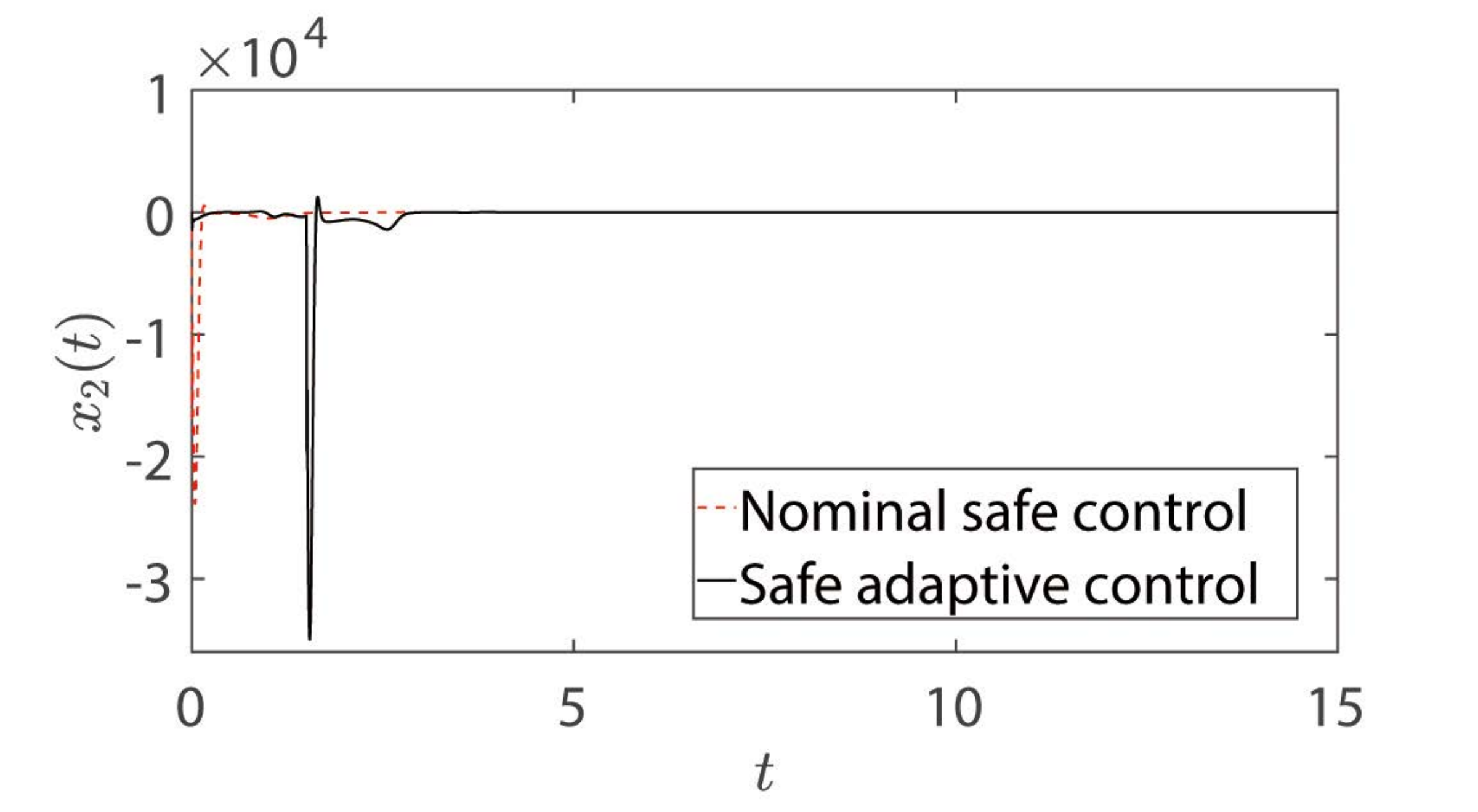}}
  \centerline{(b) $x_2(t)$}
\end{minipage}
\caption{Responses of $x_1(t), x_2(t)$ under the nominal safe
and safe adaptive controllers.}
\label{fig:x1x2}
\end{figure}

The response of the distal ODE's output state $y_1(t)$ that is expected to be kept in the safe region, i.e., the non-negativity, the other state $y_2(t)$ in the distal ODE, and the estimates of the unknown plant parameters $d_1,d_2,b$ are shown in Figs. \ref{fig:y1}, \ref{fig:y2}. In  Fig. \ref{fig:y1}, the blue dot-dash line denotes the result without control, the red dash line denotes the result with nominal safe (output-positive) control $U$, the black solid line shows the result under the safe (output-positive) adaptive controller $U_a$. Three results have the same behavior before $t=1$ s because in this time period no control action reaches the $Y$-ODE and the responses only depend on the initial values of the plant.  As compared to the nominal safe control, even though the mismatch between the parameter estimates and their true values degrades the control performance of the safe adaptive controller before around $t=2.5$ (the adaption time is 1.5 s that can be seen in b) in Fig. \ref{fig:y2},  and the time taken by the updated control actions spreading from $x=1$ to $x=0$ is 1s), the response under the safe adaptive control begins to fast converge to zero like the nominal control result after 2.5 s.  Both nominal safe and safe adaptive controllers can constrain the output state $y_1(t)$ in the safe region, i.e., $y_1(t)\ge 0$, and achieve the exponential convergence to zero of $y_1(t)$, while the state diverges in the open loop since the simulation model is open-loop unstable.
We know from Fig. \ref{fig:y2}-b that the exact identification of the unknown parameters is achieved at the first triggering time, under the nonzero initial values given in Sec. \ref{sec:simmod}. Please note that the tiny difference between the final estimates $\hat d_1,\hat d_2$ and their true values comes from the approximation error of integration mentioned at the end of the last subsection. It is shown in Figs \ref{fig:wz}, \ref{fig:x1x2} that the PDE plant states $z(x,t)$, $w(x,t)$ and the nonlinear ODE states $x_1(t),x_2(t)$ all converge to zero under the nominal safe and safe adaptive control inputs.
\section{Conclusion and Future Work}
In this paper, we present a safe adaptive control design method for
$2\times 2$ hyperbolic PDEs sandwiched between a strict-feedback nonlinear
ODE on the actuated side and a linear ODE on the uncontrolled
side. The coefficients of the unstable sources in the PDE domain and of the input signal in the distal ODE to be kept safe are unknown. Our contort design guarantees: 1) the finite-time exact parameter identification of the unknown parameters; 2) the safety of the state furthermost from the control input; 3) the exponential regulation of overall plant states to zero. The
 numerical simulation illustrates the validity of the proposed control design.

In practice, the measurement error will affect the parameter identification precision and also lead to slight fluctuation of the parameter estimates even if on $t\ge t_f$, which has not been dealt with in this paper. One way to reduce the effect of the measurement error is to increase the design parameter $T$, $\tilde N$ or postpone the switching time $t_f$, with the purpose of letting more measured data take part in estimation. This will prolong the duration of adaptive learning, during which the state runs in a conservative safe region. How to improve the robustness with respect to the aforementioned errors in practice while avoiding overly large convergence time of the parameter estimates is our future work.
\section*{Appendix}
\setcounter{subsection}{0}
\setcounter{section}{0}
\subsection{Kernel conditions}
\subsubsection{Kernel conditions in the transformation \eqref{eq:contran1a}, \eqref{eq:contran1b}}\label{sec:ker}
\setcounter{equation}{0}
\renewcommand{\theequation}{A.\arabic{equation}}
By  mapping  \eqref{eq:ZA}, \eqref{eq:o2}--\eqref{eq:o4} and \eqref{eq:targ5}--\eqref{eq:targ4} with using \eqref{eq:o1}, the conditions of the kernels $\varphi,\phi,\Psi,\Phi,\gamma$ and $\lambda$ in the backstepping transformation \eqref{eq:contran1a}, \eqref{eq:contran1b} are obtained as the following equations
\begin{align}
&{q_2}{{\varphi }_y}(x,y) - {q_1}{{\varphi }_x}(x,y) - {d_1}\phi (x,y)=0,\label{1}\\
&{  {q_1}{{\phi }_x}(x,y) {+} {q_1}{{\phi }_y}(x,y)} {+} {d_2}\varphi (x,y)=0,\\
& {q_2}{{\Psi }_x}(x,y)  - {q_1}{{\Psi }_y}(x,y)  - {d_2}\Phi (x,y)=0,\label{eq:kerf0}\\
&{{q_2}{\Phi _x}(x,y)+{q_2}{\Phi _y}(x,y)  - {d_1}\Psi (x,y)}=0,\label{eq:kerf1}
\end{align}
with the boundary conditions
\begin{align}
&\varphi (x,x)=\frac{d_1}{{q_1} + {q_2}},\label{eq:ker1}\\
&\phi (x,0)= \frac{q_2}{{q_1}p}\varphi (x,0)-\frac{1}{{q_1}p}{\gamma}(x)B,\\
&\Psi (x,x)=\frac{-d_2 }{{q_1} + {q_2}} ,\label{eq:Psi}\\
& \Phi (x,0) = \frac{{q_1}p}{{q_2}}\Psi (x,0)+ \frac{1}{{q_2}}\lambda (x)B,\label{eq:Phi0}
\end{align}
evolving in the triangular domain $\mathcal D=\{(x,y):0\le y\le x\le 1\}$, where
\begin{align}
&{\lambda (x)=K^T e^{\frac{Ax}{q_2}}},\label{eq:kerf}\\
& {{\gamma}(x)=pK^Te^{\frac{-Ax}{q_1}}}   \label{eq:ker6}
\end{align}
with $K^T$ defined in \eqref{eq:K}.
The well-posedness of the equation set \eqref{1}--\eqref{eq:Phi0} is ensured by Theorems 4 and 5 in \cite{Vazquez2011Backstepping}. Also,  we show the explicit solution of \eqref{eq:kerf0}, \eqref{eq:kerf1}, \eqref{eq:Psi}, \eqref{eq:Phi0}, that is--the explicit expression of kernels $\Psi$, $\Phi$ that are used in the control law, in Appendix\ref{AP:ex}.

\subsubsection{Kernel conditions in the inverse transformation \eqref{eq:Icontran1a},
\eqref{eq:Icontran1b}}\label{sec:inker}
By matching \eqref{eq:o2}--\eqref{eq:o5}, \eqref{eq:ZA} and \eqref{eq:targ5}--\eqref{eq:targ8} with using \eqref{eq:o1}, the conditions of kernels ${\bar\phi},\bar\varphi,\bar\gamma,\bar\Psi,\bar\Phi,\bar\lambda$ are obtained as the following equations
\begin{align}
&q_2\bar{\Phi}_x(x,y)+q_2\bar{\Phi}_y(x,y)+d_2\bar{\varphi}(x,y)=0,\label{eq:kkk1} \\
  &q_1\bar{\varphi}_x(x,y)-q_2\bar{\varphi}_y(x,y)-d_1\bar{\Phi}(x,y)=0, \\
		&q_2\bar{\Psi}_x(x,y)-q_1\bar{\Psi}_y(x,y)+d_2\bar{\phi}(x,y)=0, \\
		&q_1\bar{\phi}_x(x,y)+q_1\bar{\phi}_y(x,y)-d_1\bar{\Psi}(x,y)=0
  \end{align} with the boundary conditions
  \begin{align}
		&\bar{\varphi}(x,x)=\frac{-d_1}{q_1+q_2}, \\
  &\bar{\phi}(x,0)=\frac{q_2}{q_1p}\bar{\varphi}(x,0)-\frac{1}{q_1p}\bar{\gamma}(x)B,\\
  &\bar{\Psi}(x,x)=\frac{d_2}{q_1+q_2}, \\
&\bar{\Phi}(x,0)=\frac{pq_1}{q_2}\bar{\Psi}(x,0)+\frac{1}{q_2}\bar{\lambda}(x)B \label{eq:kkkn}
  \end{align} evolving in the triangular domain $\mathcal D=\{(x,y):0\le y\le x\le 1\}$, where $\bar{\gamma}(x), \bar{\lambda}(x)$ satisfy:
  \begin{align}
		&q_2\bar{\lambda}_x(x)-\bar{\lambda}(x)(A+BK^T)+d_2\bar{\gamma}(x)=0, \label{eq:eee1}\\
		&\bar{\lambda}(0)=-K^T, \\
		&q_1\bar{\gamma}_x(x)+\bar{\gamma}(x)(A+BK^T)-d_1\bar{\lambda}(x)=0, \\
		&\bar{\gamma}(0)=-pK^T. \label{eq:eee4}
	\end{align}
 The well-posedness of \eqref{eq:kkk1}--\eqref{eq:kkkn} is also ensured by Theorems 4 and 5 in \cite{Vazquez2011Backstepping}. The solution of the initial value problem  \eqref{eq:eee1}--\eqref{eq:eee4} is straightforward.
\subsection{Explicit kernels in the control law}\label{AP:ex}
\setcounter{equation}{0}
\renewcommand{\theequation}{B.\arabic{equation}}
To obtain the explicit control law, we obtain the explicit expression of the kernels $\Psi$, $\Phi$ that will be used in the control law in this subsection. We first propose a transformation to rewrite \eqref{eq:kerf0}, \eqref{eq:kerf1}, \eqref{eq:Psi}, \eqref{eq:Phi0} into the form of the kernel conditions in \cite{Vazquez2014Marcum}, then apply the result in \cite{Vazquez2014Marcum} and recall the proposed transformation to obtain the explicit solution $\Psi$, $\Phi$.

First, we propose a transformation
\begin{align}
&\Psi(x,y)={F}(x,y)+\int_y^x L(x,r){F}(r,y)dr\label{eq:traker1}\\
&\Phi(x,y)={H}(x,y)-L(x,y)+\int_y^x L(x,r){H}(r,y)dr\label{eq:traker2}
\end{align}
where
\begin{align}
L(x,y)=\frac{-1}{q_2}  \lambda(x-y) {B}.\label{eq:traker3}
\end{align}
Then inserting \eqref{eq:traker1}--\eqref{eq:traker3} into \eqref{eq:kerf0}, \eqref{eq:kerf1}, \eqref{eq:Psi}, \eqref{eq:Phi0}, one obtains that
the conditions of $F$, $H$ are
\begin{align}
&{q_2}{F}_x(x,y)- {q_1}{F}_y(x,y)={d_2}{H}(x,y),\label{eq:ex5}\\
&{q_2}{H}_x(x,y)+{q_2}{H}_y(x,y)= {d_1}{F}(x,y),
\end{align}
evolving in the triangular domain $\mathcal D=\{(x,y):0\le y\le x\le 1\}$ with the boundary condition
\begin{align}
{F}(x,x)=\frac{-d_2 }{{q_1} + {q_2}} ,
~{H}(x,0) = \frac{q_1}{q_2}p{F}(x,0),\label{eq:ex8}
\end{align}
where $L_x(x,y)+L_y(x,y)=0$ results from \eqref{eq:traker3} has been used.
Applying the result in \cite{Vazquez2014Marcum} where Bessel functions and the generalized Marcum $Q$-functions of
the first order are used, the explicit solution $F$, $H$  of \eqref{eq:ex5}--\eqref{eq:ex8} are obtained as
\begin{align}
&F(x,y)=\frac{-1}{p(q_1+q_2)}\bigg[\frac{d_1q_2}{pq_1}I_0\left(\frac{2\sqrt{d_1d_2}}{q_1+q_2}\sqrt{(x-y)(\frac{q_1}{q_2}x+y)}\right)
\notag\\&+\sqrt{d_1d_2\frac{x-y}{\frac{q_1}{q_2}x+y}}I_1\left(\frac{2\sqrt{d_1d_2}}{q_1+q_2}\sqrt{(x-y)(\frac{q_1}{q_2}x+y)}\right)
\notag\\&+(pd_2-\frac{d_1q_2}{pq_1})\Pi\left(\frac{pq_1d_2}{q_2}\frac{x-y}{q_1+q_2},\frac{d_1}{pq_1}\frac{q_1x+q_2y}{q_1+q_2}\right)\bigg]\label{eq:F}\\
&H(x,y)=\frac{-1}{q_1+q_2}\bigg[\frac{d_1}{p}I_0\left(\frac{2\sqrt{d_1d_2}}{q_1+q_2}\sqrt{(x-y)(\frac{q_1}{q_2}x+y)}\right)
\notag\\&+\sqrt{d_1d_2\frac{\frac{q_1}{q_2}x+y}{x-y}}I_1\left(\frac{2\sqrt{d_1d_2}}{q_1+q_2}\sqrt{(x-y)(\frac{q_1}{q_2}x+y)}\right)
\notag\\&+(\frac{pd_2q_1}{q_2}-\frac{d_1}{p})\Pi\left(\frac{pq_1d_2}{q_2}\frac{x-y}{q_1+q_2},\frac{d_1}{pq_1}\frac{q_1x+q_2y}{q_1+q_2}\right)\bigg]\label{eq:H}
\end{align}
where  $I_j(j \ge 0)$ denotes the modified Bessel function of the first kind (of order
$j$), and $\Pi(s_1,s_2)=e^{s_1+s_2}(1-s_2e^{-s_1}\int_0^1e^{-\tau s_2}I_0(2\sqrt{\tau s_1s_2})d\tau)$.

Inserting the explicit solutions $F(x,y) $ \eqref{eq:F}, $H(x,y)$ \eqref{eq:H}, $L(x,y)$ \eqref{eq:traker3} into \eqref{eq:traker1} and \eqref{eq:traker2}, the explicit solutions $\Psi$, $\Phi$ are thus obtained. Therefore, we obtain the explicit kernels in the control law as
\begin{align}
&\Psi(1,y)={F}(1,y)+\int_y^1 L(1,r){F}(r,y)dr\label{eq:exk1}\\
&\Phi(1,y)={H}(1,y)-L(1,y)+\int_y^1 L(1,r){H}(r,y)dr.\label{eq:exk2}
\end{align}
\subsection{The calculation details in the third transformation}\label{sec:odeback}
\setcounter{equation}{0}
\renewcommand{\theequation}{C.\arabic{equation}}
\textbf{Step. 1}
Considering  \eqref{eq:ncti} at $i=1$, we directly have $h_1(t)=x_1(t)-\Gamma(t)=\beta(1,t)$ in \eqref{eq:target5} with $\tau_0=0$, i.e., \eqref{eq:tau0}.

\textbf{Step. 2}
Taking the time derivative of \eqref{eq:ncti} at $i=1$, recalling \eqref{eq:o6} as well as \eqref{eq:ncti} for $i=2$, one obtains
\begin{align}
\dot h_1(t)&=\dot x_1(t)-\dot\Gamma(t)=h_2(t)+\tau_1+\dot\Gamma(t)+f_1(x_1)-\dot\Gamma(t)\notag\\&=-c_1h_1(t)+h_2(t)
\end{align}
with choosing
$\tau_1(x_1(t),\Gamma(t))=-c_1h_1(t)-f_1(x_1).$

\textbf{Step. 3}
Taking the time derivative of \eqref{eq:ncti} at $i=2$, recalling \eqref{eq:o6} as well as \eqref{eq:ncti} for $i=3$, one obtains
\begin{align}
\dot h_2(t)&=\dot x_2(t)-\dot\tau_1-\ddot\Gamma(t)\notag\\&=h_3(t)+\tau_2+\ddot\Gamma(t)+f_2(x_1,x_2)-\dot\tau_1-\ddot\Gamma(t)\notag\\&=-c_2h_2(t)+h_3(t)
\end{align}
with choosing
$\tau_2(\underline x_2(t),\underline \Gamma^{(1)}(t))=-c_2h_2(t)-f_2(\underline x_2)+\dot\tau_1=-c_2h_2(t)-f_2(\underline x_2)+\frac{\partial\tau_{1}}{\partial{x_1}}(x_{2}+f_2)+\frac{\partial\tau_{1}}{\partial{\Gamma(t)}}\dot\Gamma(t)$.

\textbf{Step. 4}
We make an induction hypothesis that we have
\begin{align}
\dot h_{i-1}(t)=-c_{i-1} h_{i-1}(t)+h_i(t)\label{eq:yi-1}
\end{align}
through \eqref{eq:ncti} by choosing
\begin{align}
&\tau_{i-1}(\underline x_{i-1}(t),\underline \Gamma^{(i-2)}(t))=-c_{i-1}h_{i-1}(t)-f_{i-1}(\underline x_{i-1})\notag\\&+\sum_{j=1}^{i-2}\bigg[\frac{\partial\tau_{i-2}}{\partial{x_j}}(x_{j+1}+f_j)+\frac{\partial\tau_{i-2}}{\partial{\Gamma^{(j-1)}(t)}}\Gamma^{(j)}(t)\bigg].\label{eq:i-1}
 \end{align}
Taking the time derivative of \eqref{eq:ncti} and recalling \eqref{eq:o6}, one obtains
\begin{align}
&\dot h_{i}(t)=\dot x_{i}(t)-\dot \tau_{i-1}-\Gamma^{(i)}(t)\notag\\&=h_{i+1}(t)+\tau_{i}(t)+\Gamma^{(i)}(t)+f_{i}(\underline x_{i}(t))-\dot \tau_{i-1}-\Gamma^{(i)}(t)\notag\\&=-c_ih_i(t)+h_{i+1}(t),\label{eq:yi}
\end{align}
by choosing $\tau_{i}$ as
\begin{align}
&\tau_{i}(\underline x_{i}(t),\underline \Gamma^{(i-1)}(t))=-c_i h_i(t)-f_{i}(\underline x_{i}(t))+\dot \tau_{i-1}\notag\\
&=-c_i h_i(t)-f_{i}(\underline x_{i}(t))\notag\\&+\sum_{j=1}^{i-1}\bigg[\frac{\partial\tau_{i-1}}{\partial{x_j}}(x_{j+1}+f_j)+\frac{\partial\tau_{i-1}}{\partial{\Gamma^{(j-1)}(t)}}\Gamma^{(j)}(t)\bigg].\label{eq:tau-i}\end{align}
Therefore, \eqref{eq:yi-1}, \eqref{eq:i-1} hold for $i$, i.e., the induction step holds.
Recalling the base cases $\tau_1$ and $\tau_2$ in Steps 2 and 3, it follows that \eqref{eq:yi}, \eqref{eq:tau-i} hold for $i=1,\cdots,m-1$. Therefore, we obtain \eqref{eq:taui}.

\textbf{Step. 5}
According to \eqref{eq:ncti} for $i=m$ and \eqref{eq:o7}, one obtains
\begin{align}
&\dot h_m(t)=\dot x_m(t)-\dot \tau_{m-1}-\Gamma^{(m)}(t)\notag\\&=f_m(\underline x_m)+{\sum_{i=0}^{m-1} \bar q_iz^{(i)}(1,t)+M^TY(t)}\notag\\&+ U(t)-\dot \tau_{m-1}-\Gamma^{(m)}(t).\label{eq:dym}
\end{align}
We thus arrive at
$\dot h_m(t)=-c_m h_m(t)$
by choosing $U(t)$ as
\begin{align}
U(t)=&\dot \tau_{m-1}-f_m(\underline x_m)-{\sum_{i=0}^{m-1} \bar q_iz^{(i)}(1,t)-M^TY(t)}\notag\\&-c_mh_m(t)+\Gamma^{(m)}(t)\notag\\
=&-\sum_{i=0}^{m-1} \bar q_iz^{(i)}(1,t)-M^TY(t)+\Gamma^{(m)}(t)\notag\\&-c_mh_m(t)-f_m(\underline x_m)\notag\\
&+\sum_{j=1}^{m-1}\bigg[\frac{\partial\tau_{m-1}}{\partial{x_j}}(x_{j+1}+f_j)+\frac{\partial\tau_{m-1}}{\partial{\Gamma^{(j-1)}(t)}}\Gamma^{(j)}(t)\bigg]\notag\\
=&\tau_m-\sum_{i=0}^{m-1} \bar q_iz^{(i)}(1,t)-M^TY(t)+\Gamma^{(m)}(t),
\end{align}
which is the control input given in \eqref{eq:U1}.
Therefore, we arrive at $h_i$ governed by \eqref{eq:targetym-1}, \eqref{eq:targetym}.
\subsection{Proof of Lemma \ref{lem:estimate}}\label{AP:convergence}
\setcounter{equation}{0}
\renewcommand{\theequation}{D.\arabic{equation}}
We propose the following claims that are used in completing the proof of Lemma \ref{lem:estimate}.
\begin{clam}\label{cl:nsQ0}
The sufficient and necessary conditions of $Q_{\bar n,1}(\mu_{i+1},t_{i+1})=0$, $Q_{\bar n,3}(\mu_{i+1},t_{i+1})=0$ for $\bar n=1,2,\ldots$, are $w[t]= 0$, $z[t]= 0$ on $t\in[\mu_{i+1},t_{i+1}]$, respectively.
The sufficient and necessary condition of $Q_{4}(\mu_{i+1},t_{i+1})=0$ is $w(0,t)= 0$ on $t\in[\mu_{i+1},t_{i+1}]$.
\end{clam}
\begin{proof}
The proof of the fact that the sufficient and necessary conditions of $Q_{\bar n,1}(\mu_{i+1},t_{i+1})=0$, $Q_{\bar n,3}(\mu_{i+1},t_{i+1})=0$ for $\bar n=1,2,\ldots$, are $w[t]= 0$, $z[t]= 0$ on $t\in[\mu_{i+1},t_{i+1}]$, respectively, is the same as the proof of Lemma 2 in \cite{JiAdaptive2021}, where the continuity of $\int_0^{1}\sin({x\pi \bar n})w(x,\tau)dx$ and $\int_0^{1}\sin({x\pi \bar n})z(x,\tau)dx$ guaranteed by Proposition \ref{Pro:1},  and the set $\{\sqrt{2}\sin(n\pi x):\bar n=1,2,\ldots\}$ being an orthonormal basis of $L^2(0,1)$ have been used.

By recalling \eqref{eq:Q4}, \eqref{eq:ga}, the fact that the sufficient and necessary condition of $Q_{4}(\mu_{i+1},t_{i+1})=0$ is $w(0,t)=0$ on $t\in[\mu_{i+1},t_{i+1}]$, is obtained straightforwardly.

The proof of Claim \ref{cl:nsQ0} is complete.\end{proof}
\begin{clam}\label{lem:theta}
For the adaptive estimates defined by \eqref{eq:adaptivelaw} based on the data in the interval $t\in[\mu_{i+1},t_{i+1}]$, the following statements hold:
1) If $w[t]$ is not identically zero and $z[t]$ is identically zero on $t\in[\mu_{i+1},t_{i+1}]$, then $\hat d_1(t_{i+1})=d_1$, $\hat d_2(t_{i+1})=\hat d_2(t_{i})$;
2) If $z[t]$ is not identically zero and $w[t]$ is identically zero on $t\in[\mu_{i+1},t_{i+1}]$, then $\hat d_2(t_{i+1})=d_2$, $\hat d_1(t_{i+1})=\hat d_1(t_{i})$;
3) If both $w[t]$ and $z[t]$ are identically zero on $t\in[\mu_{i+1},t_{i+1}]$, then $\hat d_1(t_{i+1})=\hat d_1(t_{i})$, $\hat d_2(t_{i+1})=\hat d_2(t_{i})$.
4) If $z[t]$, $w[t]$ are not identically zero for $t\in[\mu_{i+1},t_{i+1}]$, then $\hat d_1(t_{i+1})=d_1$, $\hat d_2(t_{i+1})=d_2$.
5) If $w(0,t)$ is not identically zero for $t\in[\mu_{i+1},t_{i+1}]$, then $\hat b(t_{i+1})=b$.
6) If $w(0,t)$ is  identically zero for $t\in[\mu_{i+1},t_{i+1}]$, then $\hat b(t_{i+1})=\hat b(t_{i})$.
\end{clam}
\begin{proof}
The proofs of statements  1 and 2  are very similar to the proofs of cases 1 and 2 in Lemma 3 in \cite{JiAdaptive2021}, and thus they are omitted.

The proof of statement 3 is shown as follows.
First, we define $S_i$ as
\begin{align}
S_{i}:= \bigg\{\bar\ell=(\ell_1,\ell_2)\in \Theta_1: \bar Z_{\bar n}(\mu_{i+1},t_{i+1})=\bar G_{\bar n}(\mu_{i+1},t_{i+1})\bar\ell,\bigg\},\label{eq:setadaptivelaw1}
\end{align}
where ${\Theta_1=\{\bar \ell\in \mathbb R^2:\underline d_1\le\ell_1\le \overline d_1 ,\underline d_2\le\ell_2\le \overline d_2\}}$, and where the matrices $\bar Z_{\bar n}$, $\bar G_{\bar n}$ are
\begin{align}
\bar Z_{\bar n}=[H_{\bar n,1},H_{\bar n,2}]^T,~~\bar G_{\bar n}=\left[
    \begin{array}{ccc}
      Q_{\bar n,1} & Q_{\bar n,2}\\
     Q_{\bar n,2} & Q_{\bar n,3}\\
    \end{array}
  \right]
\end{align}
In this case 3), we know $Q_{\bar n,1}(\mu_{i+1},t_{i+1})=0$,  $Q_{\bar n,2}(\mu_{i+1},t_{i+1})=0$, $Q_{\bar n,3}(\mu_{i+1},t_{i+1})=0$, $H_{\bar n,1}(\mu_{i+1},t_{i+1})=0$, $H_{\bar n,2}(\mu_{i+1},t_{i+1})=0$ for $\bar n=1,2,\ldots$ according to \eqref{eq:gq1}, \eqref{eq:gq2}, \eqref{eq:H1m}--\eqref{eq:Q3m}.
It follows that {$S_i=\Theta_1$}, and then \eqref{eq:adaptivelaw} shows that $\hat d_1(t_{i+1})=\hat d_1(t_{i})$, $\hat d_2(t_{i+1})=\hat d_2(t_{i})$.

The proof of statement 4 is shown as follows. For the set $S_i$ defined in \eqref{eq:setadaptivelaw1}, by virtue of \eqref{eq:Fer}, \eqref{eq:adaptivelaw}, if $S_i$ is a singleton, then it is nothing else but the least-squares estimate of the unknown vector of parameters $(d_1,d_2)$ on the interval $[\mu_{i+1},t_{i+1}]$, and $S_i=\{(d_1,d_2)\}$ according to \eqref{eq:Fer}, \eqref{eq:G}. We next prove by contradiction that $S_i=\{(d_1,d_2)\}$. Suppose that on the contrary $S_i\neq \{(d_1,d_2)\}$, i.e., $S_i$ defined by \eqref{eq:setadaptivelaw1} is not a singleton. Following the proof (from the beginning to the end of the left column on page 1576 of \cite{J2022Event}) of result 3  in Lemma of \cite{J2022Event}, and recalling \eqref{eq:o4}, we have that a necessary condition of the hypothesis that $S_i$ is not a singleton is that
\begin{align}
w(x,t)-\frac{1}{p} z(x,t)\equiv0,~t\in[\mu_{i+1},t_{i+1}], x\in[0,1], \label{eq:hy1}
\end{align}
where $z[t]$, $w[t]$ are not identically zero.
Using \eqref{eq:hy1} together with \eqref{eq:o2}, \eqref{eq:o3}, \eqref{eq:o4}, we obtain
$w(1,t)=w(0,\mu_{i+1})e^{\frac{ q_2{d_1}+q_1{d_2}p^2}{q_1+ q_2}(t-\mu_{i+1})}e^{\frac{{d_2}p^2-{d_1}}{-p(q_1+ q_2)}}:= \varrho(t)$,
$z(1,t)=p \varrho(t)$
for $t\in[\mu_{i+1},t_{i+1}]$. According to \eqref{eq:o5}--\eqref{eq:o7}, for the given signals $w(1,t), z(1,t)$, we can obtain the corresponding input signal in \eqref{eq:o7} as $\varsigma_1(t)=\hbar(w(1,t),z(1,t))=\hbar(\varrho(t),p\varrho(t))$, which is unique, where $\hbar$ is the mapping from $w(1,t), z(1,t)$ to $U(t)$ in \eqref{eq:o5}--\eqref{eq:o7}.
Therefore,  a necessary condition of the hypothesis \eqref{eq:hy1} is that the control input is exactly equal to
$\varsigma_1(t)$
for $t\in[\mu_{i+1},t_{i+1})$: contradiction for $t_{i+1}\le t_f$ recalling the point 1 in Remark \ref{rem:control} that ensures that the control input is not identically equal to $\varsigma_1(t)$ on $t\in[\mu_{i+1},t_{i+1})$, $t_{i+1}\le t_f$, by inserting an excitation signal. That is, the equation \eqref{eq:hy1},
which is a necessary condition of the hypothesis that $S_i$ not be a singleton, does not hold in the time intervals before $t_f$. Consequently,  $S_i$ is  a singleton, i.e., $S_i=\{(d_1,d_2)\}$, $i\in\mathbb Z_+$, $i<f$. Therefore, $\hat d_2(t_{i+1})=d_2, \hat d_1(t_{i+1})=d_1$ for $t_{i+1}\le t_f$. If $t_f$ does not exist, it follows that the statement 4 in this Claim holds for all intervals on $t\in[0,\infty)$. Next, we show that the statement 4 in this Claim holds for all intervals on $t\in[0,\infty)$ if $t_f$ exists. It is sufficient to prove that $\hat d_2(t_{i+1})=d_2, \hat d_1(t_{i+1})=d_1$ holds for $t_{i+1}> t_f$ in the case 4 in this Claim. Considering $[\hat d_1(t_{f}),\hat d_2(t_{f})]=[d_1,d_2]\in S_{f}$ recalling \eqref{eq:Fer}, where $S_i$ is defined in \eqref{eq:setadaptivelaw1}, we have $\hat d_2(t_{f+1})=\hat d_2(t_{f})=d_2, \hat d_1(t_{f+1})=\hat d_1(t_{f})=d_1$ according to \eqref{eq:adaptivelaw}. We then obtain that $\hat d_2(t_{i+1})=d_2, \hat d_1(t_{i+1})=d_1$ for all $t_{i+1}> t_f$ in case 4 via step by step analysis, that is, this statement holds for all intervals on $t\in[0,\infty)$ as well when $t_f$ exists.

The proof of statement 5 is shown as follows. Define
$\bar S_{i}:= \left\{\ell_3\in[\underline b,\overline b]: H_{3}(\mu_{i+1},t_{i+1})=Q_{4}(\mu_{i+1},t_{i+1})\ell_3\right\}$.
By virtue of \eqref{eq:Fer}, \eqref{eq:G}, \eqref{eq:adaptivelaw}, if $\bar S_i$
is a singleton, then $\hat b(t_{i+1})=b$. According to Claim \ref{cl:nsQ0}, we then have that $Q_{4}(\mu_{i+1},t_{i+1})\neq0$. It follows that $\bar S_i$ is a singleton. Therefore, we obtain that $\hat b(t_{i+1})=b$.

The proof of statement 6 is shown as follows. According to \eqref{eq:H3m}, \eqref{eq:Q4}, \eqref{eq:ga}, we have that $Q_{4}(\mu_{i+1},t_{i+1})=0$,  $H_{3}(\mu_{i+1},t_{i+1})=0$. We then obtain that the set $\bar S_i=\{0<\underline b\le \ell_3\le \overline b\}$. Recalling \eqref{eq:adaptivelaw}, we have that $\hat b(t_{i+1})=\hat b(t_{i})$.
\end{proof}
\begin{clam}\label{lem:keep}
If $\hat d_2(t_{i})=d_2$ (or $\hat d_1(t_{i})=d_1$, or $\hat b(t_{i})=b$) for certain $i\in \mathbb Z_+$, then $\hat d_2(t)=d_2$ (or $\hat d_1(t)=d_1$, or $\hat b(t)=b$, respectively) for all $t\in[t_i,+\infty)$.
\end{clam}
\begin{proof}
According to Claim \ref{lem:theta}, we have that $\hat d_2(t_{i+1})$ is equal to either $d_2$ or $\hat d_2(t_i)$. Therefore, if  $\hat d_2(t_i)=d_2$, then $\hat d_2(t)=d_2$ for all $t\in[t_i,\lim_{k\to \infty}(t_k))$. According to the sequence of time instants \eqref{eq:ri}, we know $\lim_{k\to \infty}(t_k)=+\infty$. The same is true of $\hat d_1$ and $\hat b$. The proof is complete.
\end{proof}
Now, we are close to completing the proof of Lemma \ref{lem:estimate}. According to the statements in Claim \ref{lem:theta}, there are three cases that lead to failed finite-time parameter identification: $z[t]$ (or  $w[t]$, or $w(0,t)$) is identically zero all the time. Next we prove the finite-time parameter identification is achieved, that is, the three cases that lead to failed finite-time parameter identification will not happen, by contradiction.

1) Suppose that $z[t]$ is identically zero for $t\in[0,\infty)$, then we know that $w[t]$ is identically zero for $t\in[0,\infty)$ according to \eqref{eq:o3}, \eqref{eq:o4}. Recalling  \eqref{eq:o5}--\eqref{eq:o7}, it means that $x_i$, $i=1,\cdots,m$ and the control signal are identically zero for $t\in[0,\infty)$: contradiction recalling the point 2 in Remark \ref{rem:control}. Therefore, $z[t]$ is not identically zero for $t\in[0,\infty)$.

2) Suppose that $w[t]$ is identically zero for $t\in[0,\infty)$, it implies that $z[t]$ is identically zero  for $t\in[0,\infty)$ if $d_2\neq 0$. It implies that $x_i$, $i=1,\cdots,m$ and the control signal are identically zero for $t\in[0,\infty)$ by recalling  \eqref{eq:o5}--\eqref{eq:o7}: contradiction recalling the point 3 in Remark \ref{rem:control}. If $d_2=0$, it follows from $w[t]\equiv0$ that $z[t]$ is identically zero  for $t\in[\frac{1}{q_1},\infty)$. It means that the control signal is identically zero for $t\in[\frac{1}{q_1},\infty)$: contraction, recalling the point 2 in Remark \ref{rem:control}. Therefore, $w[t]$ is not identically zero for $t\in[0,\infty)$.

3) Suppose that $w(0,t)$ is identically zero, then $Y(t)=Y(0)e^{At}$ in \eqref{eq:o1}, and $\beta(0,t)=-\lambda (0){Y}(t)$, $\alpha(0,t)=-p\lambda (0){Y}(t)$ obtained from \eqref{eq:o4}, \eqref{eq:contran1b}. Recalling  \eqref{eq:target3}, \eqref{eq:target5}, \eqref{eq:kerf}, \eqref{eq:ker6}, we have $\beta(x,t)=-K^TY(0)e^{A(t+\frac{1}{q_2}x)}$ and $\alpha(x,t)=-pK^TY(0)e^{A(t-\frac{1}{q_1}x)}$. Recalling \eqref{eq:Icontran1a}, \eqref{eq:Icontran1b}, $z (1,t)$ and $w (1,t)$ are thus determined, which implies that the control input now is exactly equal to the signal
$\varsigma_2(t)=\hbar(z (1,t),w (1,t))$, which is unique,
for all time, where $\hbar$ is the mapping from the given $w(1,t), z(1,t)$ to $U(t)$ in \eqref{eq:o5}--\eqref{eq:o7} mentioned above. This is  contradiction recalling the point 3 in Remark \ref{rem:control}, which avoids that the control input is identically equal to $\varsigma_2(t)$ by inserting an excitation signal. Therefore, $w(0,t)$ is not identically zero for $t\in[0,\infty)$.

According to the above analysis, recalling Claims \ref{lem:theta}, \ref{lem:keep}, we conclude that the estimates reach at the true values at a finite time instant $t_f=\min\{t_i: \exists t\in [0,t_i), w[t]\neq 0, z[t]\neq 0, w(0,t)\neq 0\}$, and the process of parameter identification is in fact the process of set $D_i$ being shrunk to a singleton that is equal to the true value $\theta$. Lemma \ref{lem:estimate} is obtained.

\end{document}